\documentclass{amsart}
\usepackage{amssymb,latexsym,amsmath,amscd,graphicx,graphics,epic,eepic,bm}
\usepackage[usenames,dvipsnames]{color}
\usepackage{enumerate}
\usepackage[all,knot,poly]{xy}

\newcommand{\zed}{\mathbb{Z}}

\newcommand{\C}{\mathbb{C}}

\newcommand{\qb}[2]{\genfrac{[}{]}{0pt}{}{#1}{#2}}

\newcommand{\rot}{\mathrm{rot}}

\theoremstyle{plain}
\newtheorem{theorem}{Theorem}[section]
\newtheorem{lemma}[theorem]{Lemma}

\newtheorem{corollary}[theorem]{Corollary}

\theoremstyle{definition}
\newtheorem{definition}[theorem]{Definition}

\newtheorem{question}[theorem]{Question}

\theoremstyle{remark}
\newtheorem{remark}[theorem]{Remark}

\numberwithin{equation}{section}

\begin{document}

\title[Kauffman-Vogel and Murakami-Ohtsuki-Yamada Polynomials]{On the Kauffman-Vogel and the Murakami-Ohtsuki-Yamada Graph Polynomials}

\author{Hao Wu}

\address{Department of Mathematics, The George Washington University, Monroe Hall, Room 240, 2115 G Street, NW, Washington DC 20052}

\email{haowu@gwu.edu}



\begin{abstract}
This paper consists of three parts.

First, we generalize the Jaeger Formula to express the Kauffman-Vogel graph polynomial as a state sum of the Murakami-Ohtsuki-Yamada graph polynomial. 

Then, we demonstrate that reversing the orientation and the color of a MOY graph along a simple circuit does not change the $\mathfrak{sl}(N)$ Murakami-Ohtsuki-Yamada polynomial or the $\mathfrak{sl}(N)$ homology of this MOY graph. In fact, reversing the orientation and the color of a component of a colored link only changes the $\mathfrak{sl}(N)$ homology by an overall grading shift. 

Finally, as an application of the first two parts, we prove that the $\mathfrak{so}(6)$ Kauffman polynomial is equal to the $2$-colored $\mathfrak{sl}(4)$ Reshetikhin-Turaev link polynomial, which implies that the $2$-colored $\mathfrak{sl}(4)$ link homology categorifies the $\mathfrak{so}(6)$ Kauffman polynomial.
\end{abstract}

\maketitle

\tableofcontents

\section{The Jaeger Formula of the Kauffman-Vogel Polynomial}\label{sec-Jaeger}

\subsection{The Kauffman and the HOMFLY-PT link polynomials} The Kauffman polynomial $P(K)(q,a)$ defined in \cite{Kauffman} is an invariant of unoriented framed link in $S^3$. Here, we use the following normalization of the Kauffman polynomial.
\begin{equation}\label{Kauffman-skein}
\begin{cases}
P(\setlength{\unitlength}{.75pt}
\begin{picture}(20,20)(-10,7)
\put(0,10){\circle{15}}
\end{picture}) = \frac{a-a^{-1}}{q-q^{-1}} +1 \\
P(\setlength{\unitlength}{.75pt}
\begin{picture}(20,20)(-10,7)
\put(-10,0){\line(1,1){20}}

\put(-2,12){\line(-1,1){8}}

\put(2,8){\line(1,-1){8}}

\end{picture}) - P(\setlength{\unitlength}{.75pt}
\begin{picture}(20,20)(-10,7)
\put(10,0){\line(-1,1){20}}

\put(2,12){\line(1,1){8}}

\put(-2,8){\line(-1,-1){8}}

\end{picture}) = (q-q^{-1})(P(\setlength{\unitlength}{.75pt}
\begin{picture}(20,20)(-10,7)
\qbezier(-10,0)(0,10)(-10,20)

\qbezier(10,0)(0,10)(10,20)
\end{picture}) - P(\setlength{\unitlength}{.75pt}
\begin{picture}(20,20)(-10,7)
\qbezier(-10,0)(0,10)(10,0)

\qbezier(-10,20)(0,10)(10,20)
\end{picture})) \\
P(\setlength{\unitlength}{.75pt}
\begin{picture}(20,20)(-10,7)
\put(-10,0){\line(1,1){12}}

\put(-2,12){\line(-1,1){8}}

\qbezier(2,12)(10,20)(10,10)

\qbezier(2,8)(10,0)(10,10)
\end{picture}) = a P(\setlength{\unitlength}{.75pt}
\begin{picture}(15,20)(-10,7)
\qbezier(-10,0)(10,10)(-10,20)
\end{picture})
\end{cases}
\end{equation}
The $\mathfrak{so}(N)$ Kauffman polynomial $P_{N}(K)(q)$ is defined to be the specialization 
\begin{equation}\label{Kauffman-N-def}
P_{N}(K)(q)= P(K)(q,q^{N-1}).
\end{equation}

The HOMFLY-PT polynomial $R(K)(q,a)$ defined in \cite{HOMFLY,PT} is an invariant of oriented framed link in $S^3$. Here, we use the following normalization of the HOMFLY-PT polynomial.
\begin{equation}\label{HOMFLY-skein}
\begin{cases}
R(\setlength{\unitlength}{.75pt}
\begin{picture}(20,20)(-10,7)
\put(0,10){\circle{15}}
\end{picture}) = \frac{a-a^{-1}}{q-q^{-1}} \\
R(\setlength{\unitlength}{.75pt}
\begin{picture}(20,20)(-10,7)
\put(-10,0){\vector(1,1){20}}

\put(-2,12){\vector(-1,1){8}}

\put(2,8){\line(1,-1){8}}

\end{picture}) - R(\setlength{\unitlength}{.75pt}
\begin{picture}(20,20)(-10,7)
\put(10,0){\vector(-1,1){20}}

\put(2,12){\vector(1,1){8}}

\put(-2,8){\line(-1,-1){8}}

\end{picture}) = (q-q^{-1})R(\setlength{\unitlength}{.75pt}
\begin{picture}(20,20)(-10,7)
\put(-10,20){\vector(-1,1){0}}

\put(10,20){\vector(1,1){0}}

\qbezier(-10,0)(0,10)(-10,20)

\qbezier(10,0)(0,10)(10,20)
\end{picture}) \\
R(\setlength{\unitlength}{.75pt}
\begin{picture}(20,20)(-10,7)
\put(-10,0){\line(1,1){12}}

\put(-2,12){\line(-1,1){8}}

\qbezier(2,12)(10,20)(10,10)

\qbezier(2,8)(10,0)(10,10)
\end{picture}) = a R(\setlength{\unitlength}{.75pt}
\begin{picture}(15,20)(-10,7)
\qbezier(-10,0)(10,10)(-10,20)
\end{picture})
\end{cases}
\end{equation}
The $\mathfrak{sl}(N)$ HOMFLY-PT polynomial $R_{N}(K)(q)$ is defined to be the specialization 
\begin{equation}\label{HOMFLY-N-def}
R_{N}(K)(q)= R(K)(q,q^{N}).
\end{equation}

It is easy to renormalize $P(K)(q,a)$ and $R(K)(q,a)$ to make them invariant under Reidemeister move (I) too. 

\subsection{The Jaeger Formula} The Jaeger Formula can be found in, for example, \cite{Ferrand, Kauffman-book}. Here, we give it a slightly different formulation.

Given an unoriented link diagram $D$, we call a segment of the link between two adjacent crossings an edge of this diagram $D$. An edge orientation of $D$ is an orientation of all the edges of $D$. We say that an edge orientation of $D$ is balanced if, at every crossing, two edges point inward and two edges point outward. Up to rotation, there are four possible balanced edge orientations near a crossing. (See Figure \ref{balanced-orientation-crossing-fig}.)

\begin{figure}[ht]
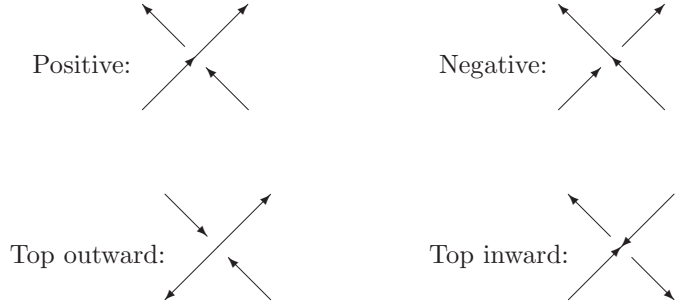

\[
\xymatrix{
\text{Positive: }\input{boc-+} && \text{Negative: }\input{boc--} \\
\text{Top outward:} \input{boc-out} && \text{Top inward:} \input{boc-in}
}
\]
\caption{Balanced edge orientations near a crossing}\label{balanced-orientation-crossing-fig} 

\end{figure}

Denote by $\widetilde{\mathcal{O}}(D)$ the set of all balanced edge orientations of $D$. Equipping $D$ with $\varrho \in \widetilde{\mathcal{O}}(D)$, we get an edge-oriented diagram $D_\varrho$. We say that $\varrho$ is admissible if $D_\varrho$ does not contain a top inward crossing. We denote by $\mathcal{O}(D)$ the subset of $\widetilde{\mathcal{O}}(D)$ consisting of all admissible balanced edge orientations of $D$.

\begin{figure}[ht]
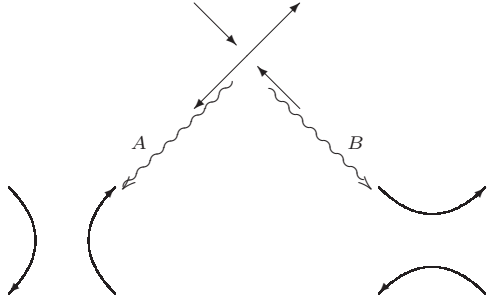

\[
\xymatrix{
& \input{boc-out} \ar@{~>}[dl]_{A} \ar@{~>}[dr]^{B} & \\
\input{res-NS-out} && \input{res-EW-out}
}
\]
\caption{Resolutions of a top outward crossing}\label{res-top-out-fig} 

\end{figure}

For $\varrho \in \mathcal{O}(D)$, we allow the two resolutions in Figure \ref{res-top-out-fig} at each top outward crossing of $D_\varrho$. A resolution $\varsigma$ of $D_\varrho$ is a choice of $A$ or $B$ resolution of every top outward crossing of $D_\varrho$. Denote by $\Sigma(D_\varrho)$ the set of all resolutions of $D_\varrho$. 

For each $\varsigma \in \Sigma(D_\varrho)$ and each top outward crossing $c$ of $D_\varrho$, we define a local weight 
\begin{equation}\label{local-weight-crossing-Jaeger}
[D_\varrho,\varsigma;c]= \begin{cases}
q-q^{-1} & \text{if } \varsigma \text{ applies } A \text{ to } c, \\
-q+q^{-1} & \text{if } \varsigma \text{ applies } B \text{ to } c.
\end{cases}
\end{equation}
The total weight $[D_\varrho,\varsigma]$ of the resolution $\varsigma$ is defined to be
\begin{equation}\label{weight-link-Jaeger}
[D_\varrho,\varsigma]= \prod_c [D_\varrho,\varsigma;c],
\end{equation}
where $c$ runs through all top outward crossings of $D_\varrho$.

For $\varsigma \in \Sigma(D_\varrho)$, denote by $D_{\varrho,\varsigma}$ the oriented link diagram (in the usual sense) obtained by applying $\varsigma$ to $D_\varrho$. As an immersed curve in $\mathbb{R}^2$, $D_{\varrho,\varsigma}$ has a rotation number $\rot(D_{\varrho,\varsigma})$ (which is also known as the Whitney index or the degree of the Gauss map.)

The following is our formulation of the Jaeger Formula, which is easily shown to be equivalent to the Jaeger Formula given in \cite{Ferrand, Kauffman-book}.

\begin{equation}\label{eq-Jaeger-formula}
P(D)(q,a^2q^{-1}) = \sum_{\varrho \in \mathcal{O}(D)} \sum_{\varsigma \in \Sigma(D_\varrho)} (a^{-1}q)^{\rot(D_{\varrho,\varsigma})} [D_\varrho,\varsigma] R(D_{\varrho,\varsigma})(q,a).
\end{equation}

Plugging $a=q^N$ into the above formula, we get
\begin{equation}\label{eq-Jaeger-formula-N}
P_{2N}(D)(q) = \sum_{\varrho \in \mathcal{O}(D)} \sum_{\varsigma \in \Sigma(D_\varrho)} q^{-(N-1)\rot(D_{\varrho,\varsigma})} [D_\varrho,\varsigma] R_N(D_{\varrho,\varsigma})(q).
\end{equation}

\subsection{The Kauffman-Vogel polynomial, the Murakami-Ohtsuki-Yamada polynomial and the Jaeger Formula}\label{subsec-Jaeger-graph} The first objective of the present paper is to generalize the Jaeger Formula \eqref{eq-Jaeger-formula} to express the Kauffman-Vogel polynomial as a state sum in terms of the Murakami-Ohtsuki-Yamada polynomial of (uncolored) oriented knotted $4$-valent graphs.

A knotted $4$-valent graph is an immersion of an abstract $4$-valent graph into $\mathbb{R}^2$ whose only singularities are finitely many crossings away from vertices. Here, a crossing is a transversal double point with one intersecting branch specified as upper and the other as lower. Two knotted $4$-valent graph are equivalent if they are isotopic to each other via a rigid vertex isotopy. (See \cite[Section 1]{KV} for the definition of rigid vertex isotopies.)

\begin{figure}[ht]
\[
\input{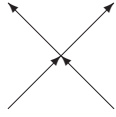}
\]
\caption{Vertex of an oriented knotted $4$-valent graph}\label{oriented-vertex-fig}

\end{figure} 

We say that a knotted $4$-valent graph is oriented if the underlying abstract $4$-valent graph is oriented in such a way that, up to rotation, very vertex in the knotted $4$-valent graph looks like the one in Figure \ref{oriented-vertex-fig}. We say that a knotted $4$-valent graph is unoriented if the underlying abstract $4$-valent graph is unoriented. Note that some orientations of the underlying abstract $4$-valent graph do not give rise to orientations of the knotted $4$-valent graph.

The Kauffman-Vogel polynomial $P(D)(q,a)$ defined in \cite{KV} is an invariant of unoriented knotted $4$-valent graphs under regular rigid vertex isotopy. It is defined by the skein relations \eqref{Kauffman-skein} of the Kauffman polynomial plus the following additional relation.
\begin{equation}\label{Kauffman-skein-vertex}
P(\setlength{\unitlength}{.75pt}
\begin{picture}(20,20)(-10,7)
\put(0,10){\line(1,1){10}}

\put(-10,0){\line(1,1){10}}

\put(0,10){\line(-1,1){10}}

\put(10,0){\line(-1,1){10}}

\end{picture})= - P(\setlength{\unitlength}{.75pt}
\begin{picture}(20,20)(-10,7)
\put(-10,0){\line(1,1){20}}

\put(-2,12){\line(-1,1){8}}

\put(2,8){\line(1,-1){8}}

\end{picture}) + q P(\setlength{\unitlength}{.75pt}
\begin{picture}(20,20)(-10,7)
\qbezier(-10,0)(0,10)(-10,20)

\qbezier(10,0)(0,10)(10,20)
\end{picture}) +q^{-1} P(\setlength{\unitlength}{.75pt}
\begin{picture}(20,20)(-10,7)
\qbezier(-10,0)(0,10)(10,0)

\qbezier(-10,20)(0,10)(10,20)
\end{picture}) = - P(\setlength{\unitlength}{.75pt}
\begin{picture}(20,20)(-10,7)
\put(10,0){\line(-1,1){20}}

\put(2,12){\line(1,1){8}}

\put(-2,8){\line(-1,-1){8}}

\end{picture}) + q^{-1} P(\setlength{\unitlength}{.75pt}
\begin{picture}(20,20)(-10,7)
\qbezier(-10,0)(0,10)(-10,20)

\qbezier(10,0)(0,10)(10,20)
\end{picture}) + q P(\setlength{\unitlength}{.75pt}
\begin{picture}(20,20)(-10,7)
\qbezier(-10,0)(0,10)(10,0)

\qbezier(-10,20)(0,10)(10,20)
\end{picture}).
\end{equation}
The $\mathfrak{so}(N)$ Kauffman-Vogel polynomial $P_{N}(D)(q)$ is defined to be the specialization 
\begin{equation}\label{Kauffman-Vogel-2N-def}
P_{N}(D)(q)= P(D)(q,q^{N-1}).
\end{equation}

The Murakami-Ohtsuki-Yamada polynomial\footnote{For oriented knotted $4$-valent graphs, the Murakami-Ohtsuki-Yamada polynomial was first defined by Kauffman and Vogel \cite{KV}. Murakami, Ohtsuki and Yamada \cite{MOY} generalized it to knotted MOY graphs and used it to recover the Reshetikhin-Turaev $\mathfrak{sl}(N)$ polynomial of links colored by wedge powers of the defining representation.} $R(D)(q,a)$ of oriented knotted $4$-valent graphs is an invariant under regular rigid vertex isotopy. It is defined by the skein relations \eqref{HOMFLY-skein} of the HOMFLY-PT polynomial plus the following additional relation.
\begin{equation}\label{HOMFLY-skein-vertex}
R(\setlength{\unitlength}{.75pt}
\begin{picture}(20,20)(-10,7)
\put(0,10){\vector(1,1){10}}

\put(-10,0){\vector(1,1){10}}

\put(0,10){\vector(-1,1){10}}

\put(10,0){\vector(-1,1){10}}

\end{picture}) = - R(\setlength{\unitlength}{.75pt}
\begin{picture}(20,20)(-10,7)
\put(-10,0){\vector(1,1){20}}

\put(-2,12){\vector(-1,1){8}}

\put(2,8){\line(1,-1){8}}

\end{picture}) + q R(\setlength{\unitlength}{.75pt}
\begin{picture}(20,20)(-10,7)
\put(-10,20){\vector(-1,1){0}}

\put(10,20){\vector(1,1){0}}

\qbezier(-10,0)(0,10)(-10,20)

\qbezier(10,0)(0,10)(10,20)
\end{picture}) = - R(\setlength{\unitlength}{.75pt}
\begin{picture}(20,20)(-10,7)
\put(10,0){\vector(-1,1){20}}

\put(2,12){\vector(1,1){8}}

\put(-2,8){\line(-1,-1){8}}

\end{picture}) + q^{-1} R(\setlength{\unitlength}{.75pt}
\begin{picture}(20,20)(-10,7)
\put(-10,20){\vector(-1,1){0}}

\put(10,20){\vector(1,1){0}}

\qbezier(-10,0)(0,10)(-10,20)

\qbezier(10,0)(0,10)(10,20)
\end{picture}).
\end{equation}
The $\mathfrak{sl}(N)$ Murakami-Ohtsuki-Yamada polynomial $R_{N}(D)(q)$ is defined to be the specialization 
\begin{equation}\label{MOY-uncolored-N-def}
R_{N}(D)(q)= R(D)(q,q^{N}).
\end{equation}

From now on, we will refer to the Kauffman-Vogel polynomial as the KV polynomial and the Murakami-Ohtsuki-Yamada polynomial as the MOY polynomial.

Given a knotted $4$-valent graph $D$, we call a segment of $D$ between two adjacent vertices or crossings an edge. (An edge can have a crossing and a vertex as its end points. Note that an edge of the underlying abstract $4$-valent graph may be divided into several edges in $D$ by crossings.) An edge orientation of $D$ is an orientation of all the edges of $D$. We say that an edge orientation of $D$ is balanced if, at every crossing and every vertex, two edges point inward and two edges point outward. As before, up to rotation, there are four possible balanced edge orientations near a crossing. (See Figure \ref{balanced-orientation-crossing-fig}.) Up to rotation, there are two possible balanced edge orientations near a vertex. (See Figure \ref{balanced-orientation-vertex-fig}.)

\begin{figure}[ht]
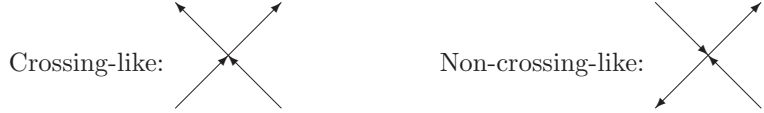

\[
\xymatrix{
\text{Crossing-like: }\input{vertex-oriented} && \text{Non-crossing-like: }\input{vertex-non-crossing} 
}
\]
\caption{Balanced edge orientations near a vertex}\label{balanced-orientation-vertex-fig} 

\end{figure}

Denote by $\widetilde{\mathcal{O}}(D)$ the set of all balanced edge orientations of $D$. Equipping $D$ with $\varrho \in \widetilde{\mathcal{O}}(D)$, we get an edge-oriented diagram $D_\varrho$. We say that $\varrho$ is admissible if $D_\varrho$ does not contain a top inward crossing. We denote by $\mathcal{O}(D)$ the subset of $\widetilde{\mathcal{O}}(D)$ consisting of all admissible balanced edge orientations of $D$.

For $\varrho \in \mathcal{O}(D)$, we allow the two resolutions in Figure \ref{res-top-out-fig} at each top outward crossing of $D_\varrho$ and the two resolutions in Figure \ref{res-non-crossing-vertex-fig} at each non-crossing-like vertex. A resolution $\varsigma$ of $D_\varrho$ is a choice of $A$ or $B$ resolution of every top outward crossing of $D_\varrho$ and $L$ or $R$ resolution of every non-crossing-like vertex. Denote by $\Sigma(D_\varrho)$ the set of all resolutions of $D_\varrho$. 

\begin{figure}[ht]
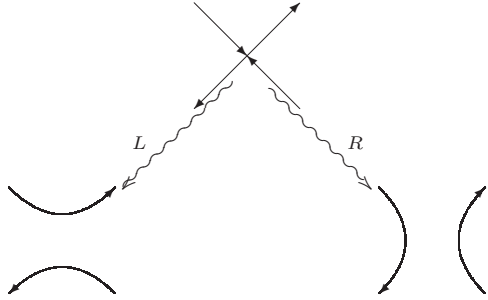

\[
\xymatrix{
& \input{vertex-non-crossing}  \ar@{~>}[dl]_{L} \ar@{~>}[dr]^{R} & \\
 \input{res-EW-out} && \input{res-NS-out}
}
\]
\caption{Resolutions of a non-crossing-like vertex}\label{res-non-crossing-vertex-fig} 

\end{figure}

Fix a $\varsigma \in \Sigma(D_\varrho)$. For each top outward crossing $c$ of $D_\varrho$, the local weight $[D_\varrho,\varsigma;c]$ is defined as in \eqref{local-weight-crossing-Jaeger}. For each non-crossing-like vertex $v$, we define a local weight $[D_\varrho,\varsigma;v]$ by the following equation.
\begin{equation}\label{local-weight-vertex-Jaeger}
[D_\varrho,\varsigma;v]= \begin{cases}
q & \text{if } \varsigma \text{ applies } L \text{ to } v, \\
q^{-1} & \text{if } \varsigma \text{ applies } R \text{ to } v.
\end{cases}
\end{equation}
The total weight $[D_\varrho,\varsigma]$ of the resolution $\varsigma$ is defined to be
\begin{equation}\label{weight-graph-Jaeger}
[D_\varrho,\varsigma]= \left(\prod_c [D_\varrho,\varsigma;c]\right) \cdot \left(\prod_v [D_\varrho,\varsigma;v]\right),
\end{equation}
where $c$ runs through all top outward crossings of $D_\varrho$ and $v$ runs through all non-crossing-like vertices of $D_\varrho$. 

The following theorem is our generalization of the Jaeger Formula to knotted $4$-valent graphs.

\begin{theorem}\label{thm-Jaeger-formula-graph}
\begin{equation}\label{eq-Jaeger-formula-graph}
P(D)(q,a^2q^{-1}) = \sum_{\varrho \in \mathcal{O}(D)} \sum_{\varsigma \in \Sigma(D_\varrho)} (a^{-1}q)^{\rot(D_{\varrho,\varsigma})} [D_\varrho,\varsigma] R(D_{\varrho,\varsigma})(q,a).
\end{equation}

Consequently, for $N\geq 1$,
\begin{equation}\label{eq-Jaeger-formula-N-graph}
P_{2N}(D)(q) = \sum_{\varrho \in \mathcal{O}(D)} \sum_{\varsigma \in \Sigma(D_\varrho)} q^{-(N-1)\rot(D_{\varrho,\varsigma})} [D_\varrho,\varsigma] R_N(D_{\varrho,\varsigma})(q).
\end{equation}
\end{theorem}

\begin{remark}
Murakami, Ohtsuki and Yamada \cite{MOY} established a state sum formula for the $\mathfrak{sl}(N)$ MOY polynomial $R_N$. Combining that with \eqref{eq-Jaeger-formula-N-graph}, we get a state sum formula for the $\mathfrak{so}(2N)$ KV polynomial. Specially, note that the $\mathfrak{sl}(1)$ MOY polynomial of a $4$-valent graph $D$ \textbf{embedded} in $\mathbb{R}^2$ is simply given by
\begin{equation}\label{eq-sl-1}
R_1(D) = \begin{cases}
1 & \text{if } D \text{ has no vertex}, \\
0 & \text{otherwise}.
\end{cases}
\end{equation}
Using \eqref{eq-sl-1} and \eqref{eq-Jaeger-formula-N-graph}, it is straightforward to recover the formula of the $\mathfrak{so}(2)$ KV polynomial of a planar $4$-valent graph given by Carpentier \cite[Theorem 4]{Carpentier1} and Caprau, Tipton \cite[Theorem 4]{Caprau-Tipton}. 

We would also like to point out that the concept of balanced edge orientation is implicit in \cite{Carpentier2}, in which Carpentier gave an alternative proof of \cite[Theorem 4]{Carpentier1}.
\end{remark}

\begin{proof}[Proof of Theorem \ref{thm-Jaeger-formula-graph}]
We prove Theorem \ref{thm-Jaeger-formula-graph} by inducting on the number of vertices in $D$. The proof comes down to a straightforward but rather lengthy tabulation of all admissible balanced edge orientations of $D$.

If $D$ contains no vertex, then \eqref{eq-Jaeger-formula-graph} becomes \eqref{eq-Jaeger-formula}, which is known to be true. Assume that \eqref{eq-Jaeger-formula-graph} is true if $D$ has at most $n-1$ vertices. Now let $D$ be a knotted $4$-valent graph with $n$ vertices. 

\begin{figure}[ht]
\[
\xymatrix{
&& \input{vertex-unoriented-D} \ar@{~>}[d] \ar@{~>}[dll] \ar@{~>}[drr] && \\
\input{cross-NE-hatD} && \input{res-NS-DA} && \input{res-EW-DB}
}
\]
\caption{}\label{D-hatD-DA-DB-fig}

\end{figure}

Choose a vertex $v$ of $D$. Define $\widehat{D}$, $D^A$ and $D^B$ to be the knotted $4$-valent graphs obtained from $D$ by replacing $v$ by the local configurations in Figure \ref{D-hatD-DA-DB-fig}. By skein relation \eqref{Kauffman-skein-vertex}, we have 
\begin{equation}\label{eq-D-hatD-DA-DB}
P(D)=-P(\widehat{D})+q P(D^A) +q^{-1} P(D^B).
\end{equation}
Note that each of $\widehat{D}$, $D^A$ and $D^B$ has only $n-1$ vertices. So \eqref{eq-Jaeger-formula-graph} is true for $\widehat{D}$, $D^A$ and $D^B$. Thus, by \eqref{eq-D-hatD-DA-DB}, to prove \eqref{eq-Jaeger-formula-graph} for $D$, we only need to check that 
\begin{eqnarray}
\label{eq-Jaeger-D-hatD-DA-DB} && \sum_{\varrho \in \mathcal{O}(D)} \sum_{\varsigma \in \Sigma(D_\varrho)} (a^{-1}q)^{\rot(D_{\varrho,\varsigma})} [D_\varrho,\varsigma] R(D_{\varrho,\varsigma})(q,a)\\
& = &  -\sum_{\varrho \in \mathcal{O}(\widehat{D})} \sum_{\varsigma \in \Sigma(\widehat{D}_\varrho)} (a^{-1}q)^{\rot(\widehat{D}_{\varrho,\varsigma})} [\widehat{D}_\varrho,\varsigma] R(\widehat{D}_{\varrho,\varsigma})(q,a) \nonumber \\
&& + q \sum_{\varrho \in \mathcal{O}(D^A)} \sum_{\varsigma \in \Sigma(D^A_\varrho)} (a^{-1}q)^{\rot(D^A_{\varrho,\varsigma})} [D^A_\varrho,\varsigma] R(D^A_{\varrho,\varsigma})(q,a) \nonumber \\
&& +q^{-1} \sum_{\varrho \in \mathcal{O}(D^B)} \sum_{\varsigma \in \Sigma(D^B_\varrho)} (a^{-1}q)^{\rot(D^B_{\varrho,\varsigma})} [D^B_\varrho,\varsigma] R(D^B_{\varrho,\varsigma})(q,a)\nonumber.
\end{eqnarray}

According to the orientations of the four edges of $D$ incidental at $v$, we divide $\mathcal{O}(D)$ into six disjoint subsets
\begin{equation}\label{O-D-Subsets}
\mathcal{O}(D)=\mathcal{O}(D;\setlength{\unitlength}{.75pt}
\begin{picture}(20,20)(-10,7)
\put(0,10){\vector(1,1){10}}

\put(-10,0){\vector(1,1){10}}

\put(0,10){\vector(-1,1){10}}

\put(10,0){\vector(-1,1){10}}

\put(4,8){\tiny{$v$}}

\end{picture})\sqcup\mathcal{O}(D;\setlength{\unitlength}{.75pt}
\begin{picture}(20,20)(-10,7)
\put(0,10){\vector(1,1){10}}

\put(-10,0){\vector(1,1){10}}

\put(-10,20){\vector(1,-1){10}}

\put(0,10){\vector(1,-1){10}}

\put(4,8){\tiny{$v$}}

\end{picture}) \sqcup\mathcal{O}(D;\setlength{\unitlength}{.75pt}
\begin{picture}(20,20)(-10,7)
\put(10,20){\vector(-1,-1){10}}

\put(0,10){\vector(-1,-1){10}}

\put(-10,20){\vector(1,-1){10}}

\put(0,10){\vector(1,-1){10}}

\put(4,8){\tiny{$v$}}

\end{picture}) \sqcup\mathcal{O}(D;\setlength{\unitlength}{.75pt}
\begin{picture}(20,20)(-10,7)
\put(10,20){\vector(-1,-1){10}}

\put(0,10){\vector(-1,-1){10}}

\put(0,10){\vector(-1,1){10}}

\put(10,0){\vector(-1,1){10}}

\put(4,8){\tiny{$v$}}

\end{picture}) \sqcup\mathcal{O}(D;\setlength{\unitlength}{.75pt}
\begin{picture}(20,20)(-10,7)
\put(0,10){\vector(1,1){10}}

\put(0,10){\vector(-1,-1){10}}

\put(-10,20){\vector(1,-1){10}}

\put(10,0){\vector(-1,1){10}}

\put(4,8){\tiny{$v$}}

\end{picture}) \sqcup\mathcal{O}(D;\setlength{\unitlength}{.75pt}
\begin{picture}(20,20)(-10,7)
\put(10,20){\vector(-1,-1){10}}

\put(-10,0){\vector(1,1){10}}

\put(0,10){\vector(-1,1){10}}

\put(0,10){\vector(1,-1){10}}

\put(4,8){\tiny{$v$}}

\end{picture}),
\end{equation}  
where $\ast$ in $\mathcal{O}(D;\ast)$ specifies the edge orientation near $v$. Note that, depending on $D$, some of these subsets may be empty. Using similar notations, we have the following partitions of $\mathcal{O}(\widehat{D})$, $\mathcal{O}(D^A)$ and $\mathcal{O}(D^B)$.
\begin{eqnarray}
\label{O-hatD-Subsets}&& \mathcal{O}(\widehat{D})= \mathcal{O}(\widehat{D};\setlength{\unitlength}{.75pt}
\begin{picture}(20,20)(-10,7)
\put(-10,0){\vector(1,1){20}}

\put(-2,12){\vector(-1,1){8}}

\put(2,8){\line(1,-1){8}}

\end{picture}) \sqcup\mathcal{O}(\widehat{D};\setlength{\unitlength}{.75pt}
\begin{picture}(20,20)(-10,7)
\put(-10,0){\vector(1,1){20}}

\put(-10,20){\line(1,-1){8}}

\put(2,8){\vector(1,-1){8}}

\end{picture}) \sqcup\mathcal{O}(\widehat{D};\setlength{\unitlength}{.75pt}
\begin{picture}(20,20)(-10,7)
\put(10,20){\vector(-1,-1){20}}

\put(-10,20){\line(1,-1){8}}

\put(2,8){\vector(1,-1){8}}

\end{picture}) \sqcup\mathcal{O}(\widehat{D};\setlength{\unitlength}{.75pt}
\begin{picture}(20,20)(-10,7)
\put(10,20){\vector(-1,-1){20}}
\put(-2,12){\vector(-1,1){8}}

\put(2,8){\line(1,-1){8}}
\end{picture}) \sqcup\mathcal{O}(\widehat{D};\setlength{\unitlength}{.75pt}
\begin{picture}(20,20)(-10,7)
\put(0,10){\vector(-1,-1){10}}

\put(0,10){\vector(1,1){10}}

\put(-10,20){\vector(1,-1){8}}

\put(10,0){\vector(-1,1){8}}
\end{picture}), \\
\label{O-DA-Subsets}&& \mathcal{O}(D^A) = \mathcal{O}(D^A;\setlength{\unitlength}{.75pt}
\begin{picture}(20,20)(-10,7)

\put(-10,20){\vector(-1,1){0}}

\put(10,20){\vector(1,1){0}}

\qbezier(-10,0)(0,10)(-10,20)

\qbezier(10,0)(0,10)(10,20)
\end{picture}) \sqcup \mathcal{O}(D^A;\setlength{\unitlength}{.75pt}
\begin{picture}(20,20)(-10,7)

\put(-10,0){\vector(-1,-1){0}}

\put(10,0){\vector(1,-1){0}}

\qbezier(-10,0)(0,10)(-10,20)

\qbezier(10,0)(0,10)(10,20)
\end{picture}) \sqcup \mathcal{O}(D^A;\setlength{\unitlength}{.75pt}
\begin{picture}(20,20)(-10,7)

\put(-10,0){\vector(-1,-1){0}}

\put(10,20){\vector(1,1){0}}

\qbezier(-10,0)(0,10)(-10,20)

\qbezier(10,0)(0,10)(10,20)
\end{picture}) \sqcup \mathcal{O}(D^A;\setlength{\unitlength}{.75pt}
\begin{picture}(20,20)(-10,7)

\put(-10,20){\vector(-1,1){0}}

\put(10,0){\vector(1,-1){0}}

\qbezier(-10,0)(0,10)(-10,20)

\qbezier(10,0)(0,10)(10,20)
\end{picture}),  \\
\label{O-DB-Subsets}&& \mathcal{O}(D^B) = \mathcal{O}(D^B;\setlength{\unitlength}{.75pt}
\begin{picture}(20,20)(-10,7)
\put(10,20){\vector(1,1){0}}

\put(10,0){\vector(1,-1){0}}

\qbezier(-10,0)(0,10)(10,0)

\qbezier(-10,20)(0,10)(10,20)
\end{picture}) \sqcup\mathcal{O}(D^B;\setlength{\unitlength}{.75pt}
\begin{picture}(20,20)(-10,7)
\put(-10,20){\vector(-1,1){0}}

\put(-10,0){\vector(-1,-1){0}}

\qbezier(-10,0)(0,10)(10,0)

\qbezier(-10,20)(0,10)(10,20)
\end{picture}) \sqcup\mathcal{O}(D^B;\setlength{\unitlength}{.75pt}
\begin{picture}(20,20)(-10,7)
\put(10,20){\vector(1,1){0}}

\put(-10,0){\vector(-1,-1){0}}

\qbezier(-10,0)(0,10)(10,0)

\qbezier(-10,20)(0,10)(10,20)
\end{picture}) \sqcup\mathcal{O}(D^B;\setlength{\unitlength}{.75pt}
\begin{picture}(20,20)(-10,7)
\put(-10,20){\vector(-1,1){0}}

\put(10,0){\vector(1,-1){0}}

\qbezier(-10,0)(0,10)(10,0)

\qbezier(-10,20)(0,10)(10,20)
\end{picture}).
\end{eqnarray}

First, let us consider the subset $\mathcal{O}(D;\setlength{\unitlength}{.75pt}
\begin{picture}(20,20)(-10,7)
\put(0,10){\vector(1,1){10}}

\put(-10,0){\vector(1,1){10}}

\put(0,10){\vector(-1,1){10}}

\put(10,0){\vector(-1,1){10}}

\put(4,8){\tiny{$v$}}

\end{picture})$. There are obvious bijections 
\begin{eqnarray*}
\mathcal{O}(D;\setlength{\unitlength}{.75pt}
\begin{picture}(20,20)(-10,7)
\put(0,10){\vector(1,1){10}}

\put(-10,0){\vector(1,1){10}}

\put(0,10){\vector(-1,1){10}}

\put(10,0){\vector(-1,1){10}}

\put(4,8){\tiny{$v$}}

\end{picture}) & \xrightarrow{\varphi} & \mathcal{O}(\widehat{D};\setlength{\unitlength}{.75pt}
\begin{picture}(20,20)(-10,7)
\put(-10,0){\vector(1,1){20}}

\put(-2,12){\vector(-1,1){8}}

\put(2,8){\line(1,-1){8}}

\end{picture}), \\
\mathcal{O}(D;\setlength{\unitlength}{.75pt}
\begin{picture}(20,20)(-10,7)
\put(0,10){\vector(1,1){10}}

\put(-10,0){\vector(1,1){10}}

\put(0,10){\vector(-1,1){10}}

\put(10,0){\vector(-1,1){10}}

\put(4,8){\tiny{$v$}}

\end{picture}) & \xrightarrow{\psi} & \mathcal{O}(D^A;\setlength{\unitlength}{.75pt}
\begin{picture}(20,20)(-10,7)

\put(-10,20){\vector(-1,1){0}}

\put(10,20){\vector(1,1){0}}

\qbezier(-10,0)(0,10)(-10,20)

\qbezier(10,0)(0,10)(10,20)
\end{picture})
\end{eqnarray*}
that preserve the orientations of corresponding edges. Moreover, for each $\varrho \in \mathcal{O}(D;\setlength{\unitlength}{.75pt}
\begin{picture}(20,20)(-10,7)
\put(0,10){\vector(1,1){10}}

\put(-10,0){\vector(1,1){10}}

\put(0,10){\vector(-1,1){10}}

\put(10,0){\vector(-1,1){10}}

\put(4,8){\tiny{$v$}}

\end{picture})$, there are obvious bijections
\begin{eqnarray*}
\Sigma(D_\varrho) \xrightarrow{\varphi_\varrho} \Sigma(\widehat{D}_{\varphi(\varrho)}), \\
\Sigma(D_\varrho) \xrightarrow{\psi_\varrho} \Sigma(D^A_{\psi(\varrho)})
\end{eqnarray*}
such that, for any $\varsigma \in \Sigma(D_\varrho)$, $\varsigma$, $\varphi_\varrho(\varsigma)$ and $\psi_\varrho(\varsigma)$ are identical outside the parts shown in Figure \ref{D-hatD-DA-DB-fig}. Note that the four edges at $v$ are oriented in a crossing-like way. So $\varsigma$ (resp. $\varphi_\varrho(\varsigma)$ and $\psi_\varrho(\varsigma)$) does not change the part of $D_\varrho$ (resp. $\widehat{D}_{\varphi(\varrho)}$ and $D^A_{\psi(\varrho)}$) shown in Figure \ref{D-hatD-DA-DB-fig}. This implies that 
\begin{equation}\label{weight-D-hatD-DA}
[D_\varrho,\varsigma] =[\widehat{D}_{\varphi(\varrho)},\varphi_\varrho(\varsigma)] =[D^A_{\psi(\varrho)},\psi_\varrho(\varsigma)].
\end{equation}
It is also easy to see that
\begin{equation}\label{rot-D-hatD-DA}
\rot(D_{\varrho,\varsigma}) =\rot(\widehat{D}_{\varphi(\varrho),\varphi_\varrho(\varsigma)}) = \rot(D^A_{\psi(\varrho),\psi_\varrho(\varsigma)}).
\end{equation}
By the skein relation \eqref{HOMFLY-skein-vertex}, we know that
\begin{equation}\label{HOMFLY-D-hatD-DA}
R(D_{\varrho,\varsigma}) = -R(\widehat{D}_{\varphi(\varrho),\varphi_\varrho(\varsigma)}) +q R(D^A_{\psi(\varrho),\psi_\varrho(\varsigma)}).
\end{equation}
Combining equations \eqref{weight-D-hatD-DA}, \eqref{rot-D-hatD-DA} and \eqref{HOMFLY-D-hatD-DA}, we get
\begin{eqnarray}
\label{eq-Jaeger-D-hatD-DA-DB-u} && \sum_{\varrho \in \mathcal{O}(D;\setlength{\unitlength}{.75pt}
\begin{picture}(20,20)(-10,7)
\put(0,10){\vector(1,1){10}}

\put(-10,0){\vector(1,1){10}}

\put(0,10){\vector(-1,1){10}}

\put(10,0){\vector(-1,1){10}}

\put(4,8){\tiny{$v$}}

\end{picture})} \sum_{\varsigma \in \Sigma(D_\varrho)} (a^{-1}q)^{\rot(D_{\varrho,\varsigma})} [D_\varrho,\varsigma] R(D_{\varrho,\varsigma})(q,a)\\
& = &  -\sum_{\varrho \in \mathcal{O}(\widehat{D};\setlength{\unitlength}{.75pt}
\begin{picture}(20,20)(-10,7)
\put(-10,0){\vector(1,1){20}}

\put(-2,12){\vector(-1,1){8}}

\put(2,8){\line(1,-1){8}}

\end{picture})} \sum_{\varsigma \in \Sigma(\widehat{D}_\varrho)} (a^{-1}q)^{\rot(\widehat{D}_{\varrho,\varsigma})} [\widehat{D}_\varrho,\varsigma] R(\widehat{D}_{\varrho,\varsigma})(q,a) \nonumber \\
&& + q \sum_{\varrho \in \mathcal{O}(D^A;\setlength{\unitlength}{.75pt}
\begin{picture}(20,20)(-10,7)

\put(-10,20){\vector(-1,1){0}}

\put(10,20){\vector(1,1){0}}

\qbezier(-10,0)(0,10)(-10,20)

\qbezier(10,0)(0,10)(10,20)
\end{picture})} \sum_{\varsigma \in \Sigma(D^A_\varrho)} (a^{-1}q)^{\rot(D^A_{\varrho,\varsigma})} [D^A_\varrho,\varsigma] R(D^A_{\varrho,\varsigma})(q,a). \nonumber
\end{eqnarray}

One can similarly deduce that
\begin{eqnarray}
\label{eq-Jaeger-D-hatD-DA-DB-r} && \sum_{\varrho \in \mathcal{O}(D;\setlength{\unitlength}{.75pt}
\begin{picture}(20,20)(-10,7)
\put(0,10){\vector(1,1){10}}

\put(-10,0){\vector(1,1){10}}

\put(-10,20){\vector(1,-1){10}}

\put(0,10){\vector(1,-1){10}}

\put(4,8){\tiny{$v$}}

\end{picture})} \sum_{\varsigma \in \Sigma(D_\varrho)} (a^{-1}q)^{\rot(D_{\varrho,\varsigma})} [D_\varrho,\varsigma] R(D_{\varrho,\varsigma})(q,a)\\
& = &  -\sum_{\varrho \in \mathcal{O}(\widehat{D};\setlength{\unitlength}{.75pt}
\begin{picture}(20,20)(-10,7)
\put(-10,0){\vector(1,1){20}}

\put(-10,20){\line(1,-1){8}}

\put(2,8){\vector(1,-1){8}}

\end{picture})} \sum_{\varsigma \in \Sigma(\widehat{D}_\varrho)} (a^{-1}q)^{\rot(\widehat{D}_{\varrho,\varsigma})} [\widehat{D}_\varrho,\varsigma] R(\widehat{D}_{\varrho,\varsigma})(q,a) \nonumber \\
&& + q^{-1} \sum_{\varrho \in \mathcal{O}(D^B;\setlength{\unitlength}{.75pt}
\begin{picture}(20,20)(-10,7)
\put(10,20){\vector(1,1){0}}

\put(10,0){\vector(1,-1){0}}

\qbezier(-10,0)(0,10)(10,0)

\qbezier(-10,20)(0,10)(10,20)
\end{picture})} \sum_{\varsigma \in \Sigma(D^B_\varrho)} (a^{-1}q)^{\rot(D^B_{\varrho,\varsigma})} [D^B_\varrho,\varsigma] R(D^B_{\varrho,\varsigma})(q,a), \nonumber
\end{eqnarray}
\begin{eqnarray}
\label{eq-Jaeger-D-hatD-DA-DB-d} && \sum_{\varrho \in \mathcal{O}(D;\setlength{\unitlength}{.75pt}
\begin{picture}(20,20)(-10,7)
\put(10,20){\vector(-1,-1){10}}

\put(0,10){\vector(-1,-1){10}}

\put(-10,20){\vector(1,-1){10}}

\put(0,10){\vector(1,-1){10}}

\put(4,8){\tiny{$v$}}

\end{picture})} \sum_{\varsigma \in \Sigma(D_\varrho)} (a^{-1}q)^{\rot(D_{\varrho,\varsigma})} [D_\varrho,\varsigma] R(D_{\varrho,\varsigma})(q,a)\\
& = &  -\sum_{\varrho \in \mathcal{O}(\widehat{D};\setlength{\unitlength}{.75pt}
\begin{picture}(20,20)(-10,7)
\put(10,20){\vector(-1,-1){20}}

\put(-10,20){\line(1,-1){8}}

\put(2,8){\vector(1,-1){8}}

\end{picture})} \sum_{\varsigma \in \Sigma(\widehat{D}_\varrho)} (a^{-1}q)^{\rot(\widehat{D}_{\varrho,\varsigma})} [\widehat{D}_\varrho,\varsigma] R(\widehat{D}_{\varrho,\varsigma})(q,a) \nonumber \\
&& + q \sum_{\varrho \in \mathcal{O}(D^A;\setlength{\unitlength}{.75pt}
\begin{picture}(20,20)(-10,7)

\put(-10,0){\vector(-1,-1){0}}

\put(10,0){\vector(1,-1){0}}

\qbezier(-10,0)(0,10)(-10,20)

\qbezier(10,0)(0,10)(10,20)
\end{picture})} \sum_{\varsigma \in \Sigma(D^A_\varrho)} (a^{-1}q)^{\rot(D^A_{\varrho,\varsigma})} [D^A_\varrho,\varsigma] R(D^A_{\varrho,\varsigma})(q,a), \nonumber
\end{eqnarray}
\begin{eqnarray}
\label{eq-Jaeger-D-hatD-DA-DB-l} && \sum_{\varrho \in \mathcal{O}(D;\setlength{\unitlength}{.75pt}
\begin{picture}(20,20)(-10,7)
\put(10,20){\vector(-1,-1){10}}

\put(0,10){\vector(-1,-1){10}}

\put(0,10){\vector(-1,1){10}}

\put(10,0){\vector(-1,1){10}}

\put(4,8){\tiny{$v$}}

\end{picture})} \sum_{\varsigma \in \Sigma(D_\varrho)} (a^{-1}q)^{\rot(D_{\varrho,\varsigma})} [D_\varrho,\varsigma] R(D_{\varrho,\varsigma})(q,a)\\
& = &  -\sum_{\varrho \in \mathcal{O}(\widehat{D};\setlength{\unitlength}{.75pt}
\begin{picture}(20,20)(-10,7)
\put(10,20){\vector(-1,-1){20}}
\put(-2,12){\vector(-1,1){8}}

\put(2,8){\line(1,-1){8}}
\end{picture})} \sum_{\varsigma \in \Sigma(\widehat{D}_\varrho)} (a^{-1}q)^{\rot(\widehat{D}_{\varrho,\varsigma})} [\widehat{D}_\varrho,\varsigma] R(\widehat{D}_{\varrho,\varsigma})(q,a) \nonumber \\
&& + q^{-1} \sum_{\varrho \in \mathcal{O}(D^B;\setlength{\unitlength}{.75pt}
\begin{picture}(20,20)(-10,7)
\put(-10,20){\vector(-1,1){0}}

\put(-10,0){\vector(-1,-1){0}}

\qbezier(-10,0)(0,10)(10,0)

\qbezier(-10,20)(0,10)(10,20)
\end{picture})} \sum_{\varsigma \in \Sigma(D^B_\varrho)} (a^{-1}q)^{\rot(D^B_{\varrho,\varsigma})} [D^B_\varrho,\varsigma] R(D^B_{\varrho,\varsigma})(q,a). \nonumber
\end{eqnarray}

Now we consider $\mathcal{O}(D;\setlength{\unitlength}{.75pt}
\begin{picture}(20,20)(-10,7)
\put(0,10){\vector(1,1){10}}

\put(0,10){\vector(-1,-1){10}}

\put(-10,20){\vector(1,-1){10}}

\put(10,0){\vector(-1,1){10}}

\put(4,8){\tiny{$v$}}

\end{picture})$. There are obvious bijections
\begin{eqnarray*}
\mathcal{O}(D;\setlength{\unitlength}{.75pt}
\begin{picture}(20,20)(-10,7)
\put(0,10){\vector(1,1){10}}

\put(0,10){\vector(-1,-1){10}}

\put(-10,20){\vector(1,-1){10}}

\put(10,0){\vector(-1,1){10}}

\put(4,8){\tiny{$v$}}

\end{picture}) & \xrightarrow{\varphi} & \mathcal{O}(\widehat{D};\setlength{\unitlength}{.75pt}
\begin{picture}(20,20)(-10,7)
\put(0,10){\vector(-1,-1){10}}

\put(0,10){\vector(1,1){10}}

\put(-10,20){\vector(1,-1){8}}

\put(10,0){\vector(-1,1){8}}
\end{picture}), \\
\mathcal{O}(D;\setlength{\unitlength}{.75pt}
\begin{picture}(20,20)(-10,7)
\put(0,10){\vector(1,1){10}}

\put(0,10){\vector(-1,-1){10}}

\put(-10,20){\vector(1,-1){10}}

\put(10,0){\vector(-1,1){10}}

\put(4,8){\tiny{$v$}}

\end{picture}) & \xrightarrow{\psi^A} & \mathcal{O}(D^A;\setlength{\unitlength}{.75pt}
\begin{picture}(20,20)(-10,7)

\put(-10,0){\vector(-1,-1){0}}

\put(10,20){\vector(1,1){0}}

\qbezier(-10,0)(0,10)(-10,20)

\qbezier(10,0)(0,10)(10,20)
\end{picture}), \\
\mathcal{O}(D;\setlength{\unitlength}{.75pt}
\begin{picture}(20,20)(-10,7)
\put(0,10){\vector(1,1){10}}

\put(0,10){\vector(-1,-1){10}}

\put(-10,20){\vector(1,-1){10}}

\put(10,0){\vector(-1,1){10}}

\put(4,8){\tiny{$v$}}

\end{picture}) & \xrightarrow{\psi^B} & \mathcal{O}(D^B;\setlength{\unitlength}{.75pt}
\begin{picture}(20,20)(-10,7)
\put(10,20){\vector(1,1){0}}

\put(-10,0){\vector(-1,-1){0}}

\qbezier(-10,0)(0,10)(10,0)

\qbezier(-10,20)(0,10)(10,20)
\end{picture})
\end{eqnarray*}
that preserve the orientations of corresponding edges. 

Given a $\varrho \in \mathcal{O}(D;\setlength{\unitlength}{.75pt}
\begin{picture}(20,20)(-10,7)
\put(0,10){\vector(1,1){10}}

\put(0,10){\vector(-1,-1){10}}

\put(-10,20){\vector(1,-1){10}}

\put(10,0){\vector(-1,1){10}}

\put(4,8){\tiny{$v$}}

\end{picture})$, there are partitions
\begin{eqnarray}
\label{partition-D-out} \Sigma(D_\varrho) & = & \Sigma^R(D_\varrho) \sqcup \Sigma^L(D_\varrho), \\
\label{partition-hatD-out} \Sigma(\widehat{D}_{\varphi(\varrho)}) & = & \Sigma^A(\widehat{D}_{\varphi(\varrho)}) \sqcup \Sigma^B(\widehat{D}_{\varphi(\varrho)})
\end{eqnarray}
according to what local resolution is applied to $v$ and the corresponding crossing in $\widehat{D}$. There are bijections
\begin{eqnarray*}
\Sigma^R(D_\varrho) & \xrightarrow{\varphi_\varrho^A} & \Sigma^A(\widehat{D}_{\varphi(\varrho)}), \\
\Sigma^L(D_\varrho) & \xrightarrow{\varphi_\varrho^B} & \Sigma^B(\widehat{D}_{\varphi(\varrho)}), \\
\Sigma^R(D_\varrho) & \xrightarrow{\psi_\varrho^A} & \Sigma(D_{\psi^A(\varrho)}^A), \\
\Sigma^L(D_\varrho) & \xrightarrow{\psi_\varrho^B} & \Sigma(D_{\psi^B(\varrho)}^B)
\end{eqnarray*}
such that the corresponding resolutions are identical outside the parts shown in Figure \ref{D-hatD-DA-DB-fig}. 

For a $\varsigma \in \Sigma^R(D_\varrho)$, it is easy to see that $D_{\varrho,\varsigma} = \widehat{D}_{\varphi(\varrho), \varphi_\varrho^A(\varsigma)} = D_{\psi^A(\varrho), \psi_\varrho^A(\varsigma)}^A$. So 
\begin{eqnarray}
\label{eq-rot-D-hatD-DA-out} \rot(D_{\varrho,\varsigma}) & = & \rot(\widehat{D}_{\varphi(\varrho), \varphi_\varrho^A(\varsigma)}) = \rot(D_{\psi^A(\varrho), \psi_\varrho^A(\varsigma)}^A), \\
\label{eq-HOMFLY-D-hatD-DA-out} R(D_{\varrho,\varsigma}) & = & R(\widehat{D}_{\varphi(\varrho), \varphi_\varrho^A(\varsigma)}) = R(D_{\psi^A(\varrho), \psi_\varrho^A(\varsigma)}^A).
\end{eqnarray}
One can also easily check that the weights satisfy
\begin{equation}\label{eq-weight-D-hatD-DA-out}
[D_{\varrho},\varsigma] = \frac{q^{-1}}{q-q^{-1}}[\widehat{D}_{\varphi(\varrho)}, \varphi_\varrho^A(\varsigma)] = q^{-1}[D_{\psi^A(\varrho)}^A, \psi_\varrho^A(\varsigma)].
\end{equation}
So
\begin{equation}\label{eq-weight-D-hatD-DA-out-2}
[D_{\varrho},\varsigma] = -[\widehat{D}_{\varphi(\varrho)}, \varphi_\varrho^A(\varsigma)] +q[D_{\psi^A(\varrho)}^A, \psi_\varrho^A(\varsigma)].
\end{equation}

Combining equations \eqref{eq-rot-D-hatD-DA-out}, \eqref{eq-HOMFLY-D-hatD-DA-out} and \eqref{eq-weight-D-hatD-DA-out-2}, we get
\begin{eqnarray}
\label{eq-sum-D-hatD-DA-out} && \sum_{\varsigma \in \Sigma^R(D_\varrho)} (a^{-1}q)^{\rot(D_{\varrho,\varsigma})} [D_\varrho,\varsigma] R(D_{\varrho,\varsigma})(q,a)\\
& = & -\sum_{\varsigma \in \Sigma^A(\widehat{D}_{\varphi(\varrho)})} (a^{-1}q)^{\rot(\widehat{D}_{{\varphi(\varrho)},\varsigma})} [\widehat{D}_{\varphi(\varrho)},\varsigma] R(\widehat{D}_{\varphi(\varrho),\varsigma})(q,a) \nonumber \\
&& + q \sum_{\varsigma \in \Sigma(D^A_{\psi^A(\varrho)})} (a^{-1}q)^{\rot(D^A_{\psi^A(\varrho),\varsigma})} [D^A_{\psi^A(\varrho)},\varsigma] R(D^A_{\psi^A(\varrho),\varsigma})(q,a). \nonumber
\end{eqnarray}
Similarly, one gets
\begin{eqnarray}
\label{eq-sum-D-hatD-DB-out} && \sum_{\varsigma \in \Sigma^L(D_\varrho)} (a^{-1}q)^{\rot(D_{\varrho,\varsigma})} [D_\varrho,\varsigma] R(D_{\varrho,\varsigma})(q,a)\\
& = & -\sum_{\varsigma \in \Sigma^B(\widehat{D}_{\varphi(\varrho)})} (a^{-1}q)^{\rot(\widehat{D}_{\varphi(\varrho),\varsigma})} [\widehat{D}_{\varphi(\varrho)},\varsigma] R(\widehat{D}_{\varphi(\varrho),\varsigma})(q,a) \nonumber \\
&& + q^{-1} \sum_{\varsigma \in \Sigma(D^B_{\psi^B(\varrho)})} (a^{-1}q)^{\rot(D^B_{\psi^B(\varrho),\varsigma})} [D^B_{\psi^B(\varrho)},\varsigma] R(D^B_{\psi^B(\varrho),\varsigma})(q,a). \nonumber
\end{eqnarray}
Equations \eqref{eq-sum-D-hatD-DA-out} and \eqref{eq-sum-D-hatD-DB-out} imply that
\begin{eqnarray}
\label{eq-Jaeger-D-hatD-DA-DB-out} && \sum_{\varrho \in \mathcal{O}(D;\setlength{\unitlength}{.75pt}
\begin{picture}(20,20)(-10,7)
\put(0,10){\vector(1,1){10}}

\put(0,10){\vector(-1,-1){10}}

\put(-10,20){\vector(1,-1){10}}

\put(10,0){\vector(-1,1){10}}

\put(4,8){\tiny{$v$}}

\end{picture})} \sum_{\varsigma \in \Sigma(D_\varrho)} (a^{-1}q)^{\rot(D_{\varrho,\varsigma})} [D_\varrho,\varsigma] R(D_{\varrho,\varsigma})(q,a)\\
& = &  -\sum_{\varrho \in \mathcal{O}(\widehat{D};\setlength{\unitlength}{.75pt}
\begin{picture}(20,20)(-10,7)
\put(0,10){\vector(-1,-1){10}}

\put(0,10){\vector(1,1){10}}

\put(-10,20){\vector(1,-1){8}}

\put(10,0){\vector(-1,1){8}}
\end{picture})} \sum_{\varsigma \in \Sigma(\widehat{D}_\varrho)} (a^{-1}q)^{\rot(\widehat{D}_{\varrho,\varsigma})} [\widehat{D}_\varrho,\varsigma] R(\widehat{D}_{\varrho,\varsigma})(q,a) \nonumber \\
&& + q \sum_{\varrho \in \mathcal{O}(D^A;\setlength{\unitlength}{.75pt}
\begin{picture}(20,20)(-10,7)

\put(-10,0){\vector(-1,-1){0}}

\put(10,20){\vector(1,1){0}}

\qbezier(-10,0)(0,10)(-10,20)

\qbezier(10,0)(0,10)(10,20)
\end{picture})} \sum_{\varsigma \in \Sigma(D^A_\varrho)} (a^{-1}q)^{\rot(D^A_{\varrho,\varsigma})} [D^A_\varrho,\varsigma] R(D^A_{\varrho,\varsigma})(q,a) \nonumber \\
&& +q^{-1} \sum_{\varrho \in \mathcal{O}(D^B;\setlength{\unitlength}{.75pt}
\begin{picture}(20,20)(-10,7)
\put(10,20){\vector(1,1){0}}

\put(-10,0){\vector(-1,-1){0}}

\qbezier(-10,0)(0,10)(10,0)

\qbezier(-10,20)(0,10)(10,20)
\end{picture})} \sum_{\varsigma \in \Sigma(D^B_\varrho)} (a^{-1}q)^{\rot(D^B_{\varrho,\varsigma})} [D^B_\varrho,\varsigma] R(D^B_{\varrho,\varsigma})(q,a)\nonumber.
\end{eqnarray}
A similar argument shows that
\begin{eqnarray}
\label{eq-Jaeger-D-hatD-DA-DB-in} && \sum_{\varrho \in \mathcal{O}(D;\setlength{\unitlength}{.75pt}
\begin{picture}(20,20)(-10,7)
\put(10,20){\vector(-1,-1){10}}

\put(-10,0){\vector(1,1){10}}

\put(0,10){\vector(-1,1){10}}

\put(0,10){\vector(1,-1){10}}

\put(4,8){\tiny{$v$}}

\end{picture})} \sum_{\varsigma \in \Sigma(D_\varrho)} (a^{-1}q)^{\rot(D_{\varrho,\varsigma})} [D_\varrho,\varsigma] R(D_{\varrho,\varsigma})(q,a)\\
& = &  q \sum_{\varrho \in \mathcal{O}(D^A;\setlength{\unitlength}{.75pt}
\begin{picture}(20,20)(-10,7)

\put(-10,20){\vector(-1,1){0}}

\put(10,0){\vector(1,-1){0}}

\qbezier(-10,0)(0,10)(-10,20)

\qbezier(10,0)(0,10)(10,20)
\end{picture})} \sum_{\varsigma \in \Sigma(D^A_\varrho)} (a^{-1}q)^{\rot(D^A_{\varrho,\varsigma})} [D^A_\varrho,\varsigma] R(D^A_{\varrho,\varsigma})(q,a) \nonumber \\
&& +q^{-1} \sum_{\varrho \in \mathcal{O}(D^B;\setlength{\unitlength}{.75pt}
\begin{picture}(20,20)(-10,7)
\put(-10,20){\vector(-1,1){0}}

\put(10,0){\vector(1,-1){0}}

\qbezier(-10,0)(0,10)(10,0)

\qbezier(-10,20)(0,10)(10,20)
\end{picture})} \sum_{\varsigma \in \Sigma(D^B_\varrho)} (a^{-1}q)^{\rot(D^B_{\varrho,\varsigma})} [D^B_\varrho,\varsigma] R(D^B_{\varrho,\varsigma})(q,a)\nonumber.
\end{eqnarray}

From partitions \eqref{O-D-Subsets}, \eqref{O-hatD-Subsets}, \eqref{O-DA-Subsets} and \eqref{O-DB-Subsets}, we know that equations \eqref{eq-Jaeger-D-hatD-DA-DB-u}, \eqref{eq-Jaeger-D-hatD-DA-DB-r}, \eqref{eq-Jaeger-D-hatD-DA-DB-d}, \eqref{eq-Jaeger-D-hatD-DA-DB-l}, \eqref{eq-Jaeger-D-hatD-DA-DB-out} and \eqref{eq-Jaeger-D-hatD-DA-DB-in} imply that \eqref{eq-Jaeger-D-hatD-DA-DB} is true. This proves that \eqref{eq-Jaeger-formula-graph} is true for $D$. So we have completed the induction and proved \eqref{eq-Jaeger-formula-graph}. Plugging $a=q^N$ into equation \eqref{eq-Jaeger-formula-graph}, we get \eqref{eq-Jaeger-formula-N-graph}.
\end{proof}

\section{Color and Orientation in the $\mathfrak{sl}(N)$ MOY Polynomial} \label{sec-MOY}

In Section \ref{sec-Jaeger}, we only discussed a very special case of the MOY graph polynomial. In this section, we review the $\mathfrak{sl}(N)$ MOY polynomial in its full generality and prove that it is invariant under certain changes of color and orientation. In fact, such invariance holds for the colored $\mathfrak{sl}(N)$ homology too.

\subsection{The $\mathfrak{sl}(N)$ MOY graph polynomial}  In this subsection, We review the $\mathfrak{sl}(N)$ MOY graph polynomial defined \cite{MOY}. Our notations and normalizations are slightly different from that used in \cite{MOY}.

\begin{figure}[ht]
\[
\xymatrix{
\input{MOY-vertex-1} && \input{MOY-vertex-2}
}
\]
\caption{}\label{fig-MOY-vertex}
\end{figure}
 
\begin{definition}\label{def-MOY}
A MOY coloring of an oriented trivalent graph is a function from the set of edges of this graph to the set of non-negative integers such that every vertex of the colored graph is of one of the two types in Figure \ref{fig-MOY-vertex}. 

A MOY graph is an oriented trivalent graph equipped with a MOY coloring \textbf{embedded} in the plane. 

A knotted MOY graph is an oriented trivalent graph equipped with a MOY coloring \textbf{immersed} in the plane such that 
\begin{itemize}
	\item the set of singularities consists of finitely many transversal double points away from vertices,
	\item at each of these transversal double points, we specify the upper- and the lower- branches (which makes it a crossing.)
\end{itemize}
\end{definition}

Fix a positive integer $N$. Define $\mathcal{N}= \{2k-N+1|k=0,1,\dots, N-1\}$ and denote by $\mathcal{P}(\mathcal{N})$ the power set of $\mathcal{N}$.

Let $\Gamma$ be a MOY graph. Denote by $E(\Gamma)$ the set of edges of $\Gamma$, by $V(\Gamma)$ the set of vertices of $\Gamma$ and by $\mathsf{c}:E(\Gamma) \rightarrow \zed_{\geq 0}$ the color function of $\Gamma$. That is, for every edge $e$ of $\Gamma$, $\mathsf{c}(e) \in \zed_{\geq 0}$ is the color of $e$. 

\begin{definition}\label{MOY-state-def}
 A state of $\Gamma$ is a function $\varphi: E(\Gamma) \rightarrow \mathcal{P}(\mathcal{N})$ such that
\begin{enumerate}[(i)]
	\item for every edge $e$ of $\Gamma$, $\#\varphi(e) = \mathsf{c}(e)$,
	\item for every vertex $v$ of $\Gamma$, as depicted in Figure \ref{fig-MOY-vertex}, we have $\varphi(e)=\varphi(e_1) \cup \varphi(e_2)$. 
\end{enumerate}
Note that (i) and (ii) imply that $\varphi(e_1) \cap \varphi(e_2)=\emptyset$.

Denote by $\mathcal{S}_N(\Gamma)$ the set of states of $\Gamma$.
\end{definition}

Define a function $\pi:\mathcal{P}(\mathcal{N}) \times \mathcal{P}(\mathcal{N}) \rightarrow \zed_{\geq 0}$ by
\begin{equation}\label{eq-def-pi}
\pi (A_1, A_2) = \# \{(a_1,a_2) \in A_1 \times A_2 ~|~ a_1>a_2\} \text{ for } A_1,~A_2 \in \mathcal{P}(\mathcal{N}).
\end{equation}

Let $\varphi$ be a state of $\Gamma$. For a vertex $v$ of $\Gamma$ (as depicted in Figure \ref{fig-MOY-vertex}), the weight of $v$ with respect to $\varphi$ is defined to be 
\begin{equation}\label{eq-weight-vertex}
\mathrm{wt}(v;\varphi) = \frac{\mathsf{c}(e_1)\mathsf{c}(e_2)}{2} - \pi(\varphi(e_1),\varphi(e_2)).
\end{equation}

Next, replace each edge $e$ of $\Gamma$ by $\mathsf{c}(e)$ parallel edges, assign to each of these new edges a different element of $\varphi(e)$ and, at every vertex, connect each pair of new edges assigned the same element of $\mathcal{N}$. This changes $\Gamma$ into a collection $\mathcal{C}_\varphi$ of embedded oriented circles, each of which is assigned an element of $\mathcal{N}$. By abusing notation, we denote by $\varphi(C)$ the element of $\mathcal{N}$ assigned to $C\in \mathcal{C}_\varphi$. Note that: 
\begin{itemize}
	\item There may be intersections between different circles in $\mathcal{C}_\varphi$. But, each circle in $\mathcal{C}_\varphi$ is embedded, that is, it has no self-intersection or self-tangency.
	\item There may be more than one way to do this. But if we view $\mathcal{C}_\varphi$ as a virtual link and the intersection points between different elements of $\mathcal{C}_\varphi$ are virtual crossings, then the above construction is unique up to purely virtual regular Reidemeister moves.
\end{itemize}
The rotation number $\mathrm{rot}(\varphi)$ of $\varphi$ is then defined to be
\begin{equation}\label{eq-rot-state}
\mathrm{rot}(\varphi) = \sum_{C\in \mathcal{C}_\varphi} \varphi(C) \mathrm{rot}(C).
\end{equation}
Note that the sum $\sum_{C\in \mathcal{C}_\varphi} \mathrm{rot}(C)$ is independent of the choice of $\varphi \in \mathcal{S}_N(\Gamma)$. We call this sum the rotation number of $\Gamma$. That is,
\begin{equation}\label{eq-rot-gamma}
\mathrm{rot}(\Gamma) := \sum_{C\in \mathcal{C}_\varphi} \mathrm{rot}(C).
\end{equation}

\begin{definition}\label{def-MOY-graph-poly}\cite{MOY}
The $\mathfrak{sl}(N)$ MOY graph polynomial of $\Gamma$ is defined to be
\begin{equation}\label{MOY-bracket-def}
\left\langle \Gamma \right\rangle_N := \begin{cases}
\sum_{\varphi \in \mathcal{S}_N(\Gamma)} \left(\prod_{v \in V(\Gamma)} q^{\mathrm{wt}(v;\varphi)}\right) q^{\mathrm{rot}(\varphi)} & \text{if } 0\leq \mathsf{c}(e) \leq N ~\forall ~e \in E(\Gamma), \\
0 & \text{otherwise}.
\end{cases}
\end{equation}

For a knotted MOY graph $D$, define the $\mathfrak{sl}(N)$ MOY polynomial $\left\langle D \right\rangle_N$ of $D$ by applying the following skein sum at every crossing of $D$. \vspace{-2pc}
\begin{equation}\label{MOY-skein-general-+}
\left\langle \setlength{\unitlength}{1pt}
\begin{picture}(40,40)(-20,0)

\put(-20,-20){\vector(1,1){40}}

\put(20,-20){\line(-1,1){15}}

\put(-5,5){\vector(-1,1){15}}

\put(-11,15){\tiny{$_m$}}

\put(9,15){\tiny{$_n$}}

\end{picture} \right\rangle_N = \sum_{k=\max\{0,m-n\}}^{m} (-1)^{m-k} q^{k-m}\left\langle \input{square-m-n-k-left-poly}\right\rangle_N,
\end{equation}
\begin{equation}\label{MOY-skein-general--}
\left\langle \setlength{\unitlength}{1pt}
\begin{picture}(40,40)(-20,0)

\put(20,-20){\vector(-1,1){40}}

\put(-20,-20){\line(1,1){15}}

\put(5,5){\vector(1,1){15}}

\put(-11,15){\tiny{$_m$}}

\put(9,15){\tiny{$_n$}}

\end{picture} \right\rangle_N = \sum_{k=\max\{0,m-n\}}^{m} (-1)^{k-m} q^{m-k}\left\langle \input{square-m-n-k-left-poly}\right\rangle_N. \vspace{2pc}
\end{equation}
\end{definition}

\begin{theorem}\cite{MOY}
$\left\langle D \right\rangle_N$ is invariant under Reidemeister (II), (III) moves and changes under Reidemeister (I) moves only by a factor of $\pm q^k$, which depends on the color of the edge involved in the Reidemeister (I) move.
\end{theorem}

As pointed out in \cite{MOY}, if $D$ is a link diagram colored by positive integers, then $\left\langle D \right\rangle_N$ is the Reshetikhin-Turaev $\mathfrak{sl}(N)$ polynomial of the link colored by corresponding wedge powers of the defining representation of $\mathfrak{sl}(N;\C)$.

\begin{figure}[ht]
\vspace{-1pc}
\[
\xymatrix{
\input{vertex-oriented} \ar@{<~>}[rr]<-3ex> && \input{wide-edge}
}
\]
\caption{}\label{4-valent-to-MOY-fig}

\end{figure}

Let $D$ be an oriented knotted $4$-valent graph as defined in Subsection \ref{subsec-Jaeger-graph}. We color all edges of $D$ by $1$ and modify its vertices as in Figure \ref{4-valent-to-MOY-fig}. This gives us a MOY graph, which we identify with $D$. Thus, $\left\langle D\right\rangle_N$ is now defined for any oriented knotted $4$-valent graph $D$. Moreover, it was established in \cite{MOY} that
\begin{equation}\label{MOY-skein-special}
\begin{cases}
\left\langle \setlength{\unitlength}{.75pt}
\begin{picture}(20,20)(-10,7)
\put(0,10){\circle{15}}
\end{picture}\right\rangle_N = \frac{q^N-q^{-N}}{q-q^{-1}} \\
\left\langle \setlength{\unitlength}{.75pt}
\begin{picture}(20,20)(-10,7)
\put(-10,0){\vector(1,1){20}}

\put(-2,12){\vector(-1,1){8}}

\put(2,8){\line(1,-1){8}}

\end{picture}\right\rangle_N - \left\langle \setlength{\unitlength}{.75pt}
\begin{picture}(20,20)(-10,7)
\put(10,0){\vector(-1,1){20}}

\put(2,12){\vector(1,1){8}}

\put(-2,8){\line(-1,-1){8}}

\end{picture}\right\rangle_N = (q-q^{-1})\left\langle \setlength{\unitlength}{.75pt}
\begin{picture}(20,20)(-10,7)
\put(-10,20){\vector(-1,1){0}}

\put(10,20){\vector(1,1){0}}

\qbezier(-10,0)(0,10)(-10,20)

\qbezier(10,0)(0,10)(10,20)
\end{picture}\right\rangle_N \\
\left\langle \setlength{\unitlength}{.75pt}
\begin{picture}(20,20)(-10,7)
\put(-10,0){\line(1,1){12}}

\put(-2,12){\line(-1,1){8}}

\qbezier(2,12)(10,20)(10,10)

\qbezier(2,8)(10,0)(10,10)
\end{picture}\right\rangle_N  = -q^{-N} \left\langle \setlength{\unitlength}{.75pt}
\begin{picture}(15,20)(-10,7)
\qbezier(-10,0)(10,10)(-10,20)
\end{picture}\right\rangle_N \\
\left\langle \setlength{\unitlength}{.75pt}
\begin{picture}(20,20)(-10,7)
\put(0,10){\vector(1,1){10}}

\put(-10,0){\vector(1,1){10}}

\put(0,10){\vector(-1,1){10}}

\put(10,0){\vector(-1,1){10}}

\end{picture}\right\rangle_N = \left\langle \setlength{\unitlength}{.75pt}
\begin{picture}(20,20)(-10,7)
\put(-10,0){\vector(1,1){20}}

\put(-2,12){\vector(-1,1){8}}

\put(2,8){\line(1,-1){8}}

\end{picture}\right\rangle_N + q^{-1} \left\langle \setlength{\unitlength}{.75pt}
\begin{picture}(20,20)(-10,7)
\put(-10,20){\vector(-1,1){0}}

\put(10,20){\vector(1,1){0}}

\qbezier(-10,0)(0,10)(-10,20)

\qbezier(10,0)(0,10)(10,20)
\end{picture}\right\rangle_N = \left\langle \setlength{\unitlength}{.75pt}
\begin{picture}(20,20)(-10,7)
\put(10,0){\vector(-1,1){20}}

\put(2,12){\vector(1,1){8}}

\put(-2,8){\line(-1,-1){8}}

\end{picture} \right\rangle_N + q \left\langle \setlength{\unitlength}{.75pt}
\begin{picture}(20,20)(-10,7)
\put(-10,20){\vector(-1,1){0}}

\put(10,20){\vector(1,1){0}}

\qbezier(-10,0)(0,10)(-10,20)

\qbezier(10,0)(0,10)(10,20)
\end{picture}\right\rangle_N
\end{cases}
\end{equation}
(Note that our normalizetion of $\left\langle D\right\rangle_N$ is different from that in \cite{MOY}. Please refer to \cite[Theorem 14.2]{Wu-color} to see how the skein relations in \cite{MOY} translate to our normalization.)

Comparing skein relation \eqref{MOY-skein-special} to skein relations \eqref{HOMFLY-skein} and \eqref{HOMFLY-skein-vertex}, one can see that
\begin{equation}\label{eq-MOY-HOMFLY}
\left\langle D \right\rangle_N = (-1)^m R_N(\overline{D}),
\end{equation}
where $D$ is an oriented knotted $4$-valent graph, $m$ is the number of crossings in $D$, and $\overline{D}$ is the oriented knotted $4$-valent graph obtained from $D$ by switching the upper- and the lower-branches at every crossing of $D$.

\subsection{Reversing the orientation and the color along a simple circuit} In the remainder of this section, we fix a positive integer $N$. 

Let $\Gamma$ be a MOY graph and $\Delta$ a simple circuit of $\Gamma$. That is, $\Delta$ is a subgraph of $\Gamma$ such that
\begin{enumerate}[(i)]
	\item $\Delta$ is a (piecewise smoothly) embedded circle in $\mathbb{R}^2$;
	\item the orientations of all edges of $\Delta$ coincide with the same orientation of this embedded circle. 
\end{enumerate}
We call the color change $k \leadsto N-k$ a reversal of color. It is easy to see that, if we reverse both the orientation and the color of the edges along $\Delta$, then we get another MOY graph $\Gamma'$. We have the following theorem.

\begin{theorem}\label{thm-oc-reverse}
\begin{equation}\label{eq-oc-reverse}
\left\langle \Gamma \right\rangle_N = \left\langle \Gamma' \right\rangle_N
\end{equation}
\end{theorem}

\begin{proof}
We prove equation \eqref{eq-oc-reverse} using a localized formulation of the state sum \eqref{MOY-bracket-def}. 

Cut each edge of $\Gamma$ at one point in its interior. This divides $\Gamma$ into a collection of neighborhoods of its vertices, each of which is a vertex with three adjacent half-edges. (See Figure \ref{fig-MOY-vertex-angles}, where $e$, $e_1$ and $e_2$ are the three half-edges.)

\begin{figure}[ht]
\vspace{-2pc}
\[
\xymatrix{
v = \input{MOY-vertex-2-angles} && \hat{v} = \input{MOY-vertex-1-angles} 
}\vspace{2pc}
\]
\caption{}\label{fig-MOY-vertex-angles}
\end{figure}

Let $\varphi \in \mathcal{S}_N(\Gamma)$. For a vertex of $\Gamma$, if it is of the form $v$ in Figure \ref{fig-MOY-vertex-angles}, we denote by $\alpha$ the directed angle from $e_1$ to $e$ and by $\beta$ the directed angle from $e_2$ to $e$. We define
\begin{eqnarray}
\label{rot-def-local-v} \rot(v;\varphi) 
& = & \frac{1}{2\pi} \int_{e}\kappa ds  \cdot \sum \varphi(e) +\frac{1}{2\pi}\left(\alpha + \int_{e_1}\kappa ds\right) \cdot \sum \varphi(e_1) \\
&& + \frac{1}{2\pi}\left(\beta + \int_{e_2}\kappa ds\right) \cdot \sum \varphi(e_2), \nonumber
\end{eqnarray}
where $\kappa$ is the signed curvature of a plane curve and $\sum A:=\sum_{a\in A} a$ for a subset $A$ of $\mathcal{N}= \{2k-N+1|k=0,1,\dots, N-1\}$.

If the vertex is of the form $\hat{v}$ in Figure \ref{fig-MOY-vertex-angles}, we denote by $\hat{\alpha}$ the directed angle from $e$ to $e_1$ and by $\hat{\beta}$ the directed angle from $e$ to $e_2$. We define
\begin{eqnarray}
\label{rot-def-local-v-prime}  \rot(\hat{v};\varphi) 
& = & \frac{1}{2\pi}\int_{e}\kappa ds \cdot \sum \varphi(e) + \frac{1}{2\pi} \left(\hat{\alpha}+\int_{e_1}\kappa ds\right) \cdot \sum \varphi(e_1)   \\
&&  + \frac{1}{2\pi}\left(\hat{\beta}+\int_{e_2}\kappa ds\right) \cdot \sum \varphi(e_2) . \nonumber
\end{eqnarray}

Using the Gauss-Bonnet Theorem, one can easily check that 
\begin{equation} \label{eq-rot-sum}
\rot(\varphi) = \sum_{v \in V(\Gamma)} \rot(v;\varphi).
\end{equation}
So, by Definition \ref{def-MOY-graph-poly}, we have
\begin{equation} \label{eq-MOY-local}
\left\langle \Gamma \right\rangle_N = \sum_{\varphi \in \mathcal{S}_N(\Gamma)}\prod_{v \in V(\Gamma)} q^{\mathrm{wt}(v;\varphi)+\rot(v;\varphi)}.
\end{equation}

Since $\Gamma'$ is obtained from $\Gamma$ by reversing the orientation and the color of the edges along $\Delta$, there are natural bijections between $V(\Gamma)$ and $V(\Gamma')$ and between $E(\Gamma)$ and $E(\Gamma')$. Basically, every vertex corresponds to itself and every edge corresponds to itself (with reversed color and orientation if the edge belongs to $\Delta$.) For a vertex $v$ of $\Gamma$, we denote by $v'$ the vertex of $\Gamma'$ corresponding to $v$. For an edge $e$ of $\Gamma$, we denote by $e'$ the edge of $\Gamma'$ corresponding to $e$. Given a $\varphi \in \mathcal{S}_N(\Gamma)$, we define $\varphi': E(\Gamma') \rightarrow \mathcal{P(N)}$ by
\begin{equation}\label{eq-varphi-prime-def}
\varphi'(e') = \begin{cases}
\varphi(e) & \text{if } e \notin E(\Delta); \\
\mathcal{N} \setminus \varphi(e) & \text{if } e \in E(\Delta).
\end{cases}
\end{equation}
It is easy to see that $\varphi' \in \mathcal{S}_N(\Gamma')$ and that $\varphi \mapsto \varphi'$ is a bijection from $\mathcal{S}_N(\Gamma)$ to $\mathcal{S}_N(\Gamma')$.

We claim that, for all $v \in V(\Gamma)$ and $\varphi \in \mathcal{S}_N(\Gamma)$,
\begin{equation}\label{eq-local-state-match}
\mathrm{wt}(v;\varphi)+ \rot(v;\varphi) = \mathrm{wt}(v';\varphi') + \rot(v';\varphi').
\end{equation}
From equation \eqref{eq-MOY-local}, one can see that equation \eqref{eq-local-state-match} implies Theorem \ref{thm-oc-reverse}.

To prove equation \eqref{eq-local-state-match}, we need to consider how the change $\Gamma \leadsto \Gamma'$ affects the vertex $v$. If $v$ is not a vertex of $\Delta$, then none of the three edges incidental at $v$ is changed. So equation \eqref{eq-local-state-match} is trivially true. If $v$ is a vertex of $\Delta$, then exactly two edges incidental at $v$ are changed, and one of these changed edge must be the edge $e$ in Figure \ref{fig-MOY-vertex-angles} (for $v$ in either form.) So, counting the choices of the form of $v$ and the choices of the other changed edge, there are four possible ways to change $v$ if $v$ is a vertex of $\Delta$. (See Figure \ref{rotation-numbers-oc-reverse-index-fig} below.) The proofs of \eqref{eq-local-state-match} in these four cases are very similar. So we only give the details for the case in Figure \ref{fig-MOY-vertex-change} and leave the other cases to the reader.

\begin{figure}[ht]
\vspace{-2pc}
\[
\xymatrix{
v = \input{MOY-vertex-2-angles} &\ar@{~>}[rr]<-5ex> && & v' = \input{MOY-vertex-2-angles-changed}
}\vspace{2pc}
\]
\caption{}\label{fig-MOY-vertex-change}
\end{figure}

First, let us consider $\mathrm{wt}(v;\varphi)$ and $\mathrm{wt}(v';\varphi')$. 
\begin{eqnarray*}
\mathrm{wt}(v';\varphi') & = & \frac{n(N-m-n)}{2} -\pi(\varphi'(e_2'),\varphi'(e')) \\
& = & \frac{n(N-m-n)}{2} - (n(N-m-n) - \pi(\varphi'(e'),\varphi'(e_2'))) \\
& = & \pi(\mathcal{N}\setminus\varphi(e),\varphi(e_2)) - \frac{n(N-m-n)}{2}.
\end{eqnarray*}
Note that $\pi(\varphi(e_1),\varphi(e_2)) + \pi(\mathcal{N}\setminus\varphi(e),\varphi(e_2)) = \pi(\mathcal{N}\setminus\varphi(e_2),\varphi(e_2))$. The above implies
\begin{eqnarray*}
&& \mathrm{wt}(v';\varphi') - \mathrm{wt}(v;\varphi) \\
& = & \pi(\mathcal{N}\setminus\varphi(e),\varphi(e_2)) - \frac{n(N-m-n)}{2} - \frac{mn}{2} + \pi(\varphi(e_1),\varphi(e_2)) \\
& = & \pi(\mathcal{N}\setminus\varphi(e_2),\varphi(e_2)) - \frac{n(N-n)}{2}
\end{eqnarray*}
Write $\varphi(e_2) =\{j_1,\dots,j_n\} \subset \mathcal{N}$, where $j_1<j_2<\cdots <j_n$. Then
\[ 
\pi(\mathcal{N}\setminus\varphi(e_2),\varphi(e_2)) = \sum_{l=1}^n [\frac{1}{2}(N-1-j_l)-(n-l)] =  \frac{n(N-n)}{2} - \frac{1}{2}\sum \varphi(e_2).
\]
Altogether, we get
\begin{equation}\label{eq-weight-change}
\mathrm{wt}(v';\varphi') = \mathrm{wt}(v;\varphi) - \frac{1}{2}\sum \varphi(e_2).
\end{equation}

Now we compare $\rot(v;\varphi)$ to $\rot(v';\varphi')$. As before, denote by $\alpha$ the directed angle from $e_1$ to $e$ and by $\beta$ the directed angle from $e_2$ to $e$. Also denote by $\gamma$ the directed angle from $e_2'$ to $e_1'$. Note that the directed angle from $e'$ to $e_1'$ is $-\alpha$. Since $\sum\mathcal{N}=0$, we have $\sum (\mathcal{N} \setminus A) = - \sum A$ for any $A\subset \mathcal{N}$. Moreover, note that reversing the orientation of a plane curve changes the sign of its curvature $\kappa$.
\begin{eqnarray*}
\rot(v';\varphi') 
& = & \frac{1}{2\pi} \int_{e_1'}\kappa ds \cdot \sum \varphi'(e_1')  +\frac{1}{2\pi}\left(-\alpha + \int_{e'}\kappa ds\right) \cdot \sum \varphi'(e')  \\
&&  + \frac{1}{2\pi}\left(\gamma + \int_{e_2'}\kappa ds\right) \cdot \sum \varphi'(e_2') \\
& = & \frac{1}{2\pi} \int_{e_1}\kappa ds \cdot \sum \varphi(e_1) +\frac{1}{2\pi}\left(\alpha + \int_{e}\kappa ds\right) \cdot \sum \varphi(e)  \\
&&  + \frac{1}{2\pi}\left(\gamma + \int_{e_2}\kappa ds\right) \cdot \sum \varphi(e_2) \\
& = & \frac{1}{2\pi} \left( \alpha+ \int_{e_1}\kappa ds\right) \cdot \sum \varphi(e_1) +\frac{1}{2\pi}\int_{e}\kappa ds \cdot \sum \varphi(e)  \\
&&  + \frac{1}{2\pi}\left(\alpha+\gamma + \int_{e_2}\kappa ds\right) \cdot \sum \varphi(e_2) \\
& = & \rot(v;\varphi) + \frac{\alpha-\beta+\gamma}{2\pi} \cdot \sum \varphi(e_2).
\end{eqnarray*}
Note that $\alpha-\beta+\gamma =\pi$. The above shows that
\begin{equation}\label{rotation-local-change}
\rot(v';\varphi') = \rot(v;\varphi) + \frac{1}{2} \sum \varphi(e_2).
\end{equation}

Equation \eqref{eq-local-state-match} follows easily from equations \eqref{eq-weight-change} and \eqref{rotation-local-change}. 
\end{proof}

\subsection{The colored $\mathfrak{sl}(N)$ homology} Theorem \ref{thm-oc-reverse} is also true for the colored $\mathfrak{sl}(N)$ homology for MOY graphs defined in \cite{Wu-color}. More precisely, reversing the orientation and the color along a simple circuit in a MOY graph does not change the homotopy type of the matrix factorization associated to this MOY graph. To prove this, we first recall some basic properties of the matrix factorizations associated to MOY graphs.

We denote by $C_N(\Gamma)$ the $\zed_2\oplus\zed$-graded matrix factorization associated a MOY graph $\Gamma$ defined in \cite[Definition 5.5]{Wu-color} and by $\hat{C}_N(D)$ the $\zed_2\oplus\zed\oplus\zed$-graded unnormalized chain complex associated to a knotted MOY graph $D$ defined in \cite[Definitions 11.4 and 11.16]{Wu-color}. Recall that:
\begin{itemize}
	\item The $\zed_2$-grading of $C_N(\Gamma)$ and $\hat{C}_N(D)$ comes from the definition of matrix factorizations and is trivial on the homology $H_N(\Gamma)$ and $\hat{H}_N(D)$ of $C_N(\Gamma)$ and $\hat{C}_N(D)$. (See \cite[Theorems 1.3 and 14.7]{Wu-color}.)
	\item The $\zed$-grading of $C_N(\Gamma)$ comes from the polynomial grading of the base ring and is called the quantum grading. The homology $H_N(\Gamma)$ inherits this quantum grading.
	\item One $\zed$-grading of $\hat{C}_N(D)$ is the quantum grading induced by the quantum grading of matrix factorizations of MOY graphs. The other $\zed$-grading of $\hat{C}_N(D)$ is the homological grading. $\hat{H}_N(D)$ inherits both of these gradings.
\end{itemize}
Also note that, for a MOY graph $\Gamma$, $C_N(\Gamma)= \hat{C}_N(\Gamma)$.

\begin{theorem}\cite{Wu-color}\label{thm-MOY-calculus}
\begin{enumerate}
  \item $\hat{C}_N( \bigcirc_m ) \simeq \C\{\qb{N}{m}\}$, where $\bigcirc_m$ is a circle colored by $m$. \vspace{-1pc}
  \item $\hat{C}_N( \setlength{\unitlength}{1pt}
\begin{picture}(50,50)(-80,20)

\put(-60,10){\vector(0,1){10}}

\put(-60,20){\vector(-1,1){20}}

\put(-60,20){\vector(1,1){10}}

\put(-50,30){\vector(-1,1){10}}

\put(-50,30){\vector(1,1){10}}

\put(-75,3){\tiny{$i+j+k$}}

\put(-55,21){\tiny{$j+k$}}

\put(-80,42){\tiny{$i$}}

\put(-60,42){\tiny{$j$}}

\put(-40,42){\tiny{$k$}}

\end{picture}) \simeq \hat{C}_N(  \setlength{\unitlength}{1pt}
\begin{picture}(50,50)(40,20)

\put(60,10){\vector(0,1){10}}

\put(60,20){\vector(1,1){20}}

\put(60,20){\vector(-1,1){10}}

\put(50,30){\vector(1,1){10}}

\put(50,30){\vector(-1,1){10}}

\put(45,3){\tiny{$i+j+k$}}

\put(38,21){\tiny{$i+j$}}

\put(80,42){\tiny{$k$}}

\put(60,42){\tiny{$j$}}

\put(40,42){\tiny{$i$}}

\end{picture})$. \vspace{-1pc}
	\item $\hat{C}_N( \input{v-vector-m+n-bubble-slide}) \simeq  \hat{C}_N(\setlength{\unitlength}{.75pt}
\begin{picture}(55,80)(-20,40)
\put(0,0){\vector(0,1){80}}
\put(5,75){\tiny{$_{m+n}$}}
\end{picture})\{\qb{m+n}{n}\} $. 
	\item $\hat{C}_N( \setlength{\unitlength}{.75pt}
\begin{picture}(60,80)(-30,40)
\put(0,0){\vector(0,1){30}}
\put(0,30){\vector(0,1){20}}
\put(0,50){\vector(0,1){30}}

\put(-1,40){\line(1,0){2}}

\qbezier(0,30)(25,20)(25,30)
\qbezier(0,50)(25,60)(25,50)
\put(25,50){\vector(0,-1){20}}

\put(5,75){\tiny{$_{m}$}}
\put(5,5){\tiny{$_{m}$}}
\put(-30,38){\tiny{$_{m+n}$}}
\put(14,60){\tiny{$_{n}$}}
\end{picture}) \simeq  \hat{C}_N( \setlength{\unitlength}{.75pt}
\begin{picture}(40,80)(-20,40)
\put(0,0){\vector(0,1){80}}
\put(5,75){\tiny{$_{m}$}}
\end{picture})\{\qb{N-m}{n}\}$. \vspace{1pc}
  \item $\hat{C}_N( \input{decomp-III-1-slide}) \simeq \hat{C}_N( \setlength{\unitlength}{.75pt}
\begin{picture}(60,60)(-30,30)

\put(-20,0){\vector(0,1){60}}

\put(20,60){\vector(0,-1){60}}

\put(-25,30){\tiny{$_1$}}

\put(22,30){\tiny{$_m$}}
\end{picture}) \oplus  \hat{C}_N( \setlength{\unitlength}{.75pt}
\begin{picture}(60,60)(100,30)

\put(110,0){\vector(1,1){20}}

\put(130,20){\vector(1,-1){20}}

\put(130,40){\vector(0,-1){20}}

\put(130,40){\vector(-1,1){20}}

\put(150,60){\vector(-1,-1){20}}

\put(105,0){\tiny{$_1$}}

\put(105,55){\tiny{$_1$}}

\put(152,0){\tiny{$_m$}}

\put(152,55){\tiny{$_m$}}

\put(132,30){\tiny{$_{m-1}$}}

\end{picture})\{[N-m-1]\}.$  \vspace{-1pc}
  \item $\hat{C}_N( \input{decomp-IV-1-slide}) \simeq  \hat{C}_N( \setlength{\unitlength}{.75pt}
\begin{picture}(85,90)(-30,45)

\put(-20,0){\vector(0,1){45}}

\put(-20,45){\vector(0,1){45}}

\put(20,0){\vector(0,1){45}}

\put(20,45){\vector(0,1){45}}

\put(20,45){\vector(-1,0){40}}

\put(-27,20){\tiny{$_1$}}

\put(23,20){\tiny{$_{m+l-1}$}}

\put(-27,65){\tiny{$_l$}}

\put(23,65){\tiny{$_m$}}

\put(-5,38){\tiny{$_{l-1}$}}

\end{picture}) \{\qb{m-1}{n}\} \oplus  \hat{C}_N( \setlength{\unitlength}{.75pt}
\begin{picture}(60,90)(110,45)

\put(110,0){\vector(2,3){20}}

\put(150,0){\vector(-2,3){20}}

\put(130,30){\vector(0,1){30}}

\put(130,60){\vector(-2,3){20}}

\put(130,60){\vector(2,3){20}}

\put(117,20){\tiny{$_1$}}

\put(140,20){\tiny{$_{m+l-1}$}}

\put(117,65){\tiny{$_l$}}

\put(140,65){\tiny{$_m$}}

\put(133,42){\tiny{$_{m+l}$}}
\end{picture}) \{\qb{m-1}{n-1}\}.$ \vspace{1pc}
  \item $\hat{C}_N(\input{decomp-V-1-slide}) \simeq \bigoplus_{j=\max\{m-n,0\}}^m  \hat{C}_N( \input{decomp-V-2-slide})\{\qb{l}{k-j}\}$.  \vspace{2.5pc}
\end{enumerate}
Here, ``$\simeq$" means that there is a homogeneous homotopy equivalence of chain complexes of graded matrix factorizations between the two sides that preserves the $\zed_2$-grading, the quantum grading and the homological grading.

The above relations remain true if we reverse the orientation of the MOY graphs on both sides or reverse the orientation of $\mathbb{R}^2$.

For a MOY graph $\Gamma$, denote by $H_N^j(\Gamma)$ the homogeneous part of $H_N(\Gamma)$ of quantum degree $j$. Then 
\[
\sum_j q^j \dim_\C H_N^j(\Gamma) = \left\langle \Gamma \right\rangle_N.
\]
In other words, the graded dimension of $H_N(\Gamma)$ is equal to $\left\langle \Gamma \right\rangle_N$.
\end{theorem}

\begin{theorem}\cite{Wu-color}\label{thm-MOY-knotted-invariance}
\begin{enumerate}
  \item $\hat{C}_N(\setlength{\unitlength}{1pt}
\begin{picture}(40,20)(-20,20)

\put(0,0){\line(0,1){8}}

\put(0,12){\vector(0,1){8}}

\put(0,20){\vector(1,1){20}}

\put(0,20){\vector(-1,1){20}}

\put(-20,10){\vector(1,0){40}}

\put(-13,35){\tiny{$_m$}}

\put(10,35){\tiny{$_l$}}

\put(3,2){\tiny{$_{m+l}$}}

\put(10,13){\tiny{$_n$}}

\end{picture}) \simeq \hat{C}_N(\setlength{\unitlength}{1pt}
\begin{picture}(40,20)(-20,20)

\put(0,0){\vector(0,1){20}}

\put(0,20){\line(1,1){8}}

\put(0,20){\line(-1,1){8}}

\put(12,32){\vector(1,1){8}}

\put(-12,32){\vector(-1,1){8}}

\put(-20,30){\vector(1,0){40}}

\put(-13,35){\tiny{$_m$}}

\put(10,35){\tiny{$_l$}}

\put(3,2){\tiny{$_{m+l}$}}

\put(12,26){\tiny{$_n$}}

\end{picture})\vspace{1.5pc}$. We call the difference between the two knotted MOY graphs here a fork sliding. The homotopy equivalence remains true if we reverse the orientation of one or both strands involved, or if the horizontal strand is under the vertex instead of above it. See the full statement in \cite[Theorem 12.1]{Wu-color}.
  \item For a knotted MOY graph $D$, the homotopy type of $\hat{C}_N(D)$, with its three gradings, is invariant under regular Reidemeister moves.
  \item $\hat{C}_N(\setlength{\unitlength}{.75pt}
\begin{picture}(30,20)(-10,7)
\put(-10,0){\line(1,1){12}}

\put(-2,12){\line(-1,1){8}}

\qbezier(2,12)(10,20)(10,10)

\qbezier(2,8)(10,0)(10,10)

\put(12,12){\tiny{$m$}}

\end{picture}) = \hat{C}_N(\setlength{\unitlength}{.75pt}
\begin{picture}(20,20)(-10,7)
\qbezier(-10,0)(10,10)(-10,20)

\put(3,12){\tiny{$m$}}
\end{picture})\left\langle m\right\rangle \| m \| \{q^{-m(N+1-m)}\}$
  \item For a knotted MOY graph $D$, define the homology $\hat{H}_N(D)$ of $\hat{C}_N(D)$ as in \cite[Subsection 1.2]{Wu-color}. Then the graded Euler characteristic of $\hat{H}_N(D)$ is equal to $\left\langle D \right\rangle_N$.
\end{enumerate}
\end{theorem}

\begin{remark}\label{homology-grading-conventions}
In the above theorems,
\begin{enumerate}[1.]
	\item ``$\left\langle \ast \right\rangle$" means shifting the $\zed_2$-grading by $\ast$. (See for example \cite[Subsection 2.3]{Wu-color}.)
	\item ``$\|\ast\|$" means shifting the homological grading up by $\ast$. (See for example \cite[Definition 2.33]{Wu-color}.)
	\item ``$\{F(q)\}$" means shifting the quantum grading up by $F(q)$. (See for example \cite[Subsection 2.1]{Wu-color}.)
	\item Our normalization of the quantum integers is
\[
[j] := \frac{q^j-q^{-j}}{q-q^{-1}}, 
\]
\[
[j]! :=  [1] \cdot [2] \cdots [j],
\]
\[
\qb{j}{k} := \frac{[j]!}{[k]!\cdot [j-k]!}.
\]
\end{enumerate}
\end{remark}

From the definition of $\hat{C}_N(D)$ in \cite[Definitions 11.4 and 11.16]{Wu-color}, we have the following simple lemma.

\begin{lemma}\cite[Lemma 7.3]{Wu-color-ras}\label{lemma-l-N-crossings} \vspace{-1pc}
\begin{equation}\label{eq-l-N-+}
\hat{C}_N (\setlength{\unitlength}{1pt}
\begin{picture}(40,40)(-20,0)

\put(-20,-20){\vector(1,1){40}}

\put(20,-20){\line(-1,1){15}}

\put(-5,5){\vector(-1,1){15}}

\put(-11,15){\tiny{$_l$}}

\put(8,15){\tiny{$_N$}}

\end{picture}) \cong \hat{C}_N (\setlength{\unitlength}{.5pt}
\begin{picture}(85,45)(-40,45)

\put(-20,0){\vector(0,1){45}}

\put(-20,45){\vector(0,1){45}}

\put(20,0){\vector(0,1){45}}

\put(20,45){\vector(0,1){45}}

\put(-20,45){\vector(1,0){40}}

\put(-35,20){\tiny{$_N$}}

\put(25,20){\tiny{$_{l}$}}

\put(-32,65){\tiny{$_l$}}

\put(25,65){\tiny{$_N$}}

\put(-13,38){\tiny{$_{N-l}$}}

\end{picture})\|l\|\{q^{-l}\}, 
\end{equation}
\begin{equation}\label{eq-l-N--}
\hat{C}_N (\setlength{\unitlength}{1pt}
\begin{picture}(40,40)(-20,0)

\put(20,-20){\vector(-1,1){40}}

\put(-20,-20){\line(1,1){15}}

\put(5,5){\vector(1,1){15}}

\put(-11,15){\tiny{$_l$}}

\put(8,15){\tiny{$_N$}}

\end{picture}) \cong \hat{C}_N (\setlength{\unitlength}{.5pt}
\begin{picture}(85,45)(-40,45)

\put(-20,0){\vector(0,1){45}}

\put(-20,45){\vector(0,1){45}}

\put(20,0){\vector(0,1){45}}

\put(20,45){\vector(0,1){45}}

\put(-20,45){\vector(1,0){40}}

\put(-35,20){\tiny{$_N$}}

\put(25,20){\tiny{$_{l}$}}

\put(-32,65){\tiny{$_l$}}

\put(25,65){\tiny{$_N$}}

\put(-13,38){\tiny{$_{N-l}$}}

\end{picture})\|-l\|\{q^{l}\}, \vspace{1.5pc}
\end{equation}
Consequently, \vspace{-1pc}
\begin{equation}\label{eq-l-N--to+}
\hat{C}_N (\setlength{\unitlength}{1pt}
\begin{picture}(40,40)(-20,0)

\put(-20,-20){\vector(1,1){40}}

\put(20,-20){\line(-1,1){15}}

\put(-5,5){\vector(-1,1){15}}

\put(-11,15){\tiny{$_l$}}

\put(8,15){\tiny{$_N$}}

\end{picture}) \cong \hat{C}_N (\setlength{\unitlength}{1pt}
\begin{picture}(40,40)(-20,0)

\put(20,-20){\vector(-1,1){40}}

\put(-20,-20){\line(1,1){15}}

\put(5,5){\vector(1,1){15}}

\put(-11,15){\tiny{$_l$}}

\put(8,15){\tiny{$_N$}}

\end{picture})\|2l\|\{q^{-2l}\}, 
\end{equation}
\begin{equation}\label{eq-l-N-+to-}
\hat{C}_N (\setlength{\unitlength}{1pt}
\begin{picture}(40,40)(-20,0)

\put(20,-20){\vector(-1,1){40}}

\put(-20,-20){\line(1,1){15}}

\put(5,5){\vector(1,1){15}}

\put(-11,15){\tiny{$_l$}}

\put(8,15){\tiny{$_N$}}

\end{picture}) \cong \hat{C}_N (\setlength{\unitlength}{1pt}
\begin{picture}(40,40)(-20,0)

\put(-20,-20){\vector(1,1){40}}

\put(20,-20){\line(-1,1){15}}

\put(-5,5){\vector(-1,1){15}}

\put(-11,15){\tiny{$_l$}}

\put(8,15){\tiny{$_N$}}

\end{picture})\|-2l\|\{q^{2l}\}. \vspace{1.5pc}
\end{equation}
\end{lemma}

We will also need the following lemma.

\begin{lemma}\label{lemma-twisted-forks} 
$~$\vspace{-1pc}
\begin{equation}\label{eq-twisted-forks-+}
\hat{C}_N(\setlength{\unitlength}{1pt}
\begin{picture}(40,40)(-20,20)

\put(0,0){\vector(0,1){10}}

\put(5,3){\tiny{$m+n$}}

\qbezier(0,10)(-20,20)(0,30)

\put(0,30){\vector(2,1){20}}

\qbezier(0,10)(20,20)(4,28)

\put(15,32){\tiny{$n$}}

\put(-4,32){\vector(-2,1){16}}

\put(-18,32){\tiny{$m$}}

\end{picture}) \simeq \hat{C}_N(\setlength{\unitlength}{1pt}
\begin{picture}(40,40)(-20,20)

\put(0,0){\vector(0,1){20}}

\put(5,7){\tiny{$m+n$}}

\put(0,20){\vector(1,1){20}}

\put(12,25){\tiny{$n$}}

\put(0,20){\vector(-1,1){20}}

\put(-15,25){\tiny{$m$}}

\end{picture})\{q^{mn}\},
\end{equation}
\begin{equation}\label{eq-twisted-forks--}
\hat{C}_N(\setlength{\unitlength}{1pt}
\begin{picture}(40,40)(-20,20)

\put(0,0){\vector(0,1){10}}

\put(5,3){\tiny{$m+n$}}

\qbezier(0,10)(20,20)(0,30)

\put(0,30){\vector(-2,1){20}}

\qbezier(0,10)(-20,20)(-4,28)

\put(15,32){\tiny{$n$}}

\put(4,32){\vector(2,1){16}}

\put(-18,32){\tiny{$m$}}

\end{picture}) \simeq \hat{C}_N(\setlength{\unitlength}{1pt}
\begin{picture}(40,40)(-20,20)

\put(0,0){\vector(0,1){20}}

\put(5,7){\tiny{$m+n$}}

\put(0,20){\vector(1,1){20}}

\put(12,25){\tiny{$n$}}

\put(0,20){\vector(-1,1){20}}

\put(-15,25){\tiny{$m$}}

\end{picture})\{q^{-mn}\}. \vspace{2pc}
\end{equation}
The above relations remain true if we reverse the orientation of the knotted MOY graphs on both sides.
\end{lemma}

\begin{proof}
We prove \eqref{eq-twisted-forks-+} only. The proof of \eqref{eq-twisted-forks--} is similar and left to the reader. 

To prove \eqref{eq-twisted-forks-+}, we induce on $n$. When $n=1$, \eqref{eq-twisted-forks-+} is proved in \cite[Proposition 6.1]{Yonezawa3}. Assume \eqref{eq-twisted-forks-+} is true for $n$. By Theorems \ref{thm-MOY-calculus}, \ref{thm-MOY-knotted-invariance} and the induction hypothesis, we have \vspace{-1pc}
\begin{eqnarray*}
&& \hat{C}_N(\setlength{\unitlength}{1pt}
\begin{picture}(40,40)(-20,20)

\put(0,0){\vector(0,1){10}}

\put(5,3){\tiny{$m+n+1$}}

\qbezier(0,10)(-20,20)(0,30)

\put(0,30){\vector(2,1){20}}

\qbezier(0,10)(20,20)(4,28)

\put(15,32){\tiny{$n+1$}}

\put(-4,32){\vector(-2,1){16}}

\put(-18,32){\tiny{$m$}}

\end{picture})\{[n+1]\} \simeq \hat{C}_N(\input{MOY-vertex-twisted-induction-2}) \simeq \hat{C}_N(\input{MOY-vertex-twisted-induction-3}) \\
&& \\
& \simeq & \hat{C}_N(\input{MOY-vertex-twisted-induction-4}) \simeq \hat{C}_N(\input{MOY-vertex-twisted-induction-5})\{q^m\} \simeq \hat{C}_N(\input{MOY-vertex-twisted-induction-6})\{q^m\} \\
&& \\
& \simeq & \hat{C}_N(\input{MOY-vertex-twisted-induction-7})\{q^{m(n+1)}\} \simeq \hat{C}_N(\input{MOY-vertex-twisted-induction-8})\{q^{m(n+1)}\}   \\
& \simeq & \hat{C}_N(\setlength{\unitlength}{1pt}
\begin{picture}(40,40)(-20,20)

\put(0,0){\vector(0,1){20}}

\put(5,7){\tiny{$m+n+1$}}

\put(0,20){\vector(1,1){20}}

\put(14,30){\tiny{$n+1$}}

\put(0,20){\vector(-1,1){20}}

\put(-15,25){\tiny{$m$}}

\end{picture})\{q^{m(n+1)}[n+1]\}. 
\end{eqnarray*} \vspace{1pc}

\noindent By \cite[Proposition 3.20]{Wu-color}, the above implies that \vspace{-1pc}
\[
\hat{C}_N(\setlength{\unitlength}{1pt}
\begin{picture}(40,40)(-20,20)

\put(0,0){\vector(0,1){10}}

\put(5,3){\tiny{$m+n+1$}}

\qbezier(0,10)(-20,20)(0,30)

\put(0,30){\vector(2,1){20}}

\qbezier(0,10)(20,20)(4,28)

\put(15,32){\tiny{$n+1$}}

\put(-4,32){\vector(-2,1){16}}

\put(-18,32){\tiny{$m$}}

\end{picture}) \simeq  \hat{C}_N(\setlength{\unitlength}{1pt}
\begin{picture}(40,40)(-20,20)

\put(0,0){\vector(0,1){20}}

\put(5,7){\tiny{$m+n+1$}}

\put(0,20){\vector(1,1){20}}

\put(14,30){\tiny{$n+1$}}

\put(0,20){\vector(-1,1){20}}

\put(-15,25){\tiny{$m$}}

\end{picture}) \{q^{m(n+1)}\}. \vspace{1.5pc}
\] 
This completes the induction and proves \eqref{eq-twisted-forks-+}.
\end{proof}

From \cite[Theorems 1.3 and 14.7]{Wu-color}, we know that the $\zed_2$-grading of $C_N$ and $\hat{C}_N$ is always pure and does not carry any significant information. So we do not keep track of this $\zed_2$-grading in the remainder of this paper. The next is the main theorem of this subsection.

\begin{theorem}\label{thm-oc-reverse-homology}
Let $\Gamma$ be a MOY graph and $\Delta$ a simple circuit of $\Gamma$. Denote by $\Gamma'$ the MOY graph obtained from $\Gamma$ by reversing the orientation and the color of edges along $\Delta$. Then, up to an overall shift of the $\zed_2$-grading, we have
\begin{equation}\label{eq-oc-reverse-homology}
C_N(\Gamma) \simeq C_N(\Gamma').
\end{equation}
\end{theorem}

\begin{proof}
We construct the homotopy equivalence in \eqref{eq-oc-reverse-homology} in three steps.

\emph{Step One: Modifying vertices.} Let $v$ be a vertex in $\Delta$. We demonstrate in Figure \ref{oc-reverse-vertex-fig} how to modify $v$ into its corresponding vertex $v'$ in $\Gamma'$ using homogeneous homotopy equivalence. Here, we assume that the half edges $e$ and $e_1$ belong to $\Delta$.\footnote{As mentioned in the proof of Theorem \ref{thm-oc-reverse}, depending on the type of $v$ (splitting or merging) and the choice of the two edges belonging to $\Delta$, there are four possible local configurations of $\Delta$ near $v$. (See Figure \ref{rotation-numbers-oc-reverse-index-fig} below.) Figure \ref{oc-reverse-vertex-fig} covers only one of these four possible local configurations of $\Delta$ near $v$. But the other three possibilities are obtained from this one type by a reversal of orientation or horizontal flipping or both. We leave the construction for the other three cases to the reader.}

\begin{figure}[ht]
\vspace{-1pc}
\[
\xymatrix{
\input{oc-reverse-vertex-0}  \ar@{~>}[rr] && \input{oc-reverse-vertex-1} \ar@{~>}[rr] && \input{oc-reverse-vertex-2} \ar@{~>}[lllld]\\
\input{oc-reverse-vertex-3} \ar@{~>}[rr] && \input{oc-reverse-vertex-4}
}
\]
\caption{}\label{oc-reverse-vertex-fig}

\end{figure}

Note that:
\begin{enumerate}
	\item By Theorems \ref{thm-MOY-calculus}, \ref{thm-MOY-knotted-invariance} and Lemmas \ref{lemma-l-N-crossings}, \ref{lemma-twisted-forks}, each change in Figure \ref{oc-reverse-vertex-fig} induces a homogeneous homotopy equivalence.
	\item The upper right vertex in the last step in Figure \ref{oc-reverse-vertex-fig} is identical to the vertex $v'$ in $\Gamma$ corresponding to $v$. 
\end{enumerate}

\emph{Step Two: Modifying edges.} After applying Step One to every vertex along $\Delta$, each edge $e$ along $\Delta$ becomes one of the two configurations in the second row in Figure \ref{oc-reverse-edge-fig}. We further modify these two configurations as in Figure \ref{oc-reverse-edge-fig}. 

\begin{figure}[ht]
\vspace{-1pc}
\[
\xymatrix{
&& \input{h-vector-m-} \ar@{-->}[drr]^{\text{Case 2}} \ar@{-->}[dll]_{\text{Case 1}} && \\
\input{oc-reverse-edge-1} \ar@{~>}[d] && \text{or} && \input{oc-reverse-edge-2} \ar@{~>}[d] \\
\input{oc-reverse-edge-3}  &&  && \input{oc-reverse-edge-4} \ar@{~>}[d] \\
&&&&  \input{oc-reverse-edge-5}
}
\]
\caption{}\label{oc-reverse-edge-fig}

\end{figure}

Note that:

\begin{enumerate}
	\item By Theorem \ref{thm-MOY-calculus} and Lemmas \ref{lemma-l-N-crossings}, each change made to these to configurations in Figure \ref{oc-reverse-edge-fig} induces a homogeneous homotopy equivalence.
	\item At every crossing, the branch colored by $N$ is on top. 
\end{enumerate}

\emph{Step Three: Removing the unknot.} After applying Step Two to every edge along $\Delta$, we get a knotted MOY graph $D$ consisting of $\Gamma'$ and an unknot colored by $N$ that is everywhere above $\Gamma'$. We can move this unknot away from $\Gamma'$ using regular Reidemeister moves and fork sliding (given in Part (1) of Theorem \ref{thm-MOY-knotted-invariance}) and obtain a MOY graph $\widetilde{\Gamma}$. By Theorem \ref{thm-MOY-knotted-invariance}, these moves induce a homogeneous homotopy equivalence. By Part (1) of Theorem \ref{thm-MOY-calculus}, we know that removing this unknot from $\widetilde{\Gamma}$ induces a homogeneous homotopy equivalence.

Putting \textit{Steps One--Three} together, we get a homogeneous homotopy equivalence\footnote{Strictly speaking, there are two notions of homotopy equivalence involved here. That is, homotopy equivalence of matrix factorizations and homotopy equivalence of complexes of matrix factorizations. But it is easy to see that, for $C_N(\Gamma)$ and $C_N(\Gamma')$, these two notions are equivalent.} from $C_N(\Gamma)$ to $C_N(\Gamma')$. It remains to check that this homotopy equivalence preserves the quantum grading. But, by Theorem \ref{thm-MOY-calculus}, the graded dimensions of the homology of $C_N(\Gamma)$ and $C_N(\Gamma')$ are equal to $\left\langle \Gamma\right\rangle_N$ and $\left\langle \Gamma'\right\rangle_N$, which are equal by Theorem \ref{thm-oc-reverse}. So any homotopy equivalence from $C_N(\Gamma)$ to $C_N(\Gamma')$ must preserve the quantum grading. This completes the proof.
\end{proof}

Using similar techniques, we can prove that reversing the orientation and the color of a component of a link colored by elements of $\{0,1,\dots,N\}$ only changes the $\mathfrak{sl}(N)$ link homology by a grading shift.

Let $L$ be an oriented framed link in $S^3$ colored by elements of $\{0,1,\dots,N\}$. Denote by $\mathcal{K}$ the set of components of $L$ and by $\mathsf{c}: \mathcal{K} \rightarrow \{0,1,\dots,N\}$ the color function of $L$. That is, for any component $K$ of $L$, the color of $K$ is $\mathsf{c}(K) \in \{0,1,\dots,N\}$. Furthermore, for any component $K$ of $L$, denote by $w(K)$ the writhe of $K$ and, for any two components $K, ~K'$ of $L$, denote by $l(K,K')$ the linking number of $K, ~K'$.

\begin{theorem}\label{thm-oc-reverse-homology-link}
Suppose $K$ is a component of $L$ and the colored framed oriented link $L'$ is obtained from $L$ by reversing the orientation and the color of $K$. Then, up to an overall shift of the $\zed_2$-grading, 
\begin{equation}\label{eq-oc-reverse-homology-link}
\hat{C}_N(L') \simeq \hat{C}_N(L) ~\| s \|~ \{ q^{-s}\},
\end{equation}
where
\[
s = (N-2\mathsf{c}(K))w(K) - 2\sum_{K' \in \mathcal{K}\setminus \{K\}} \mathsf{c}(K') l(K,K').
\]
In particular, 
\begin{equation}\label{eq-oc-reverse-poly-link}
\left\langle L' \right\rangle_N =  (-1)^{N w(K)} \cdot q^{-s} \cdot \left\langle L \right\rangle_N.
\end{equation}
\end{theorem}

\begin{figure}[ht]
\vspace{-1pc}
\[
\xymatrix{
\input{h-vector-m-} \ar@{~>}[rr] && \input{h-vector-m-bubble}  \ar@{~>}[rr] && \input{oc-reverse-edge-2} \ar@{~>}[dllll] \\
\input{oc-reverse-edge-4} \ar@{~>}[rr] && \input{oc-reverse-edge-5} &&
}
\]
\caption{}\label{oc-reverse-homology-link-fig1}

\end{figure}

\begin{proof}
Suppose the color of $K$ is $m$. In a small segment of $K$, create a ``bubble" as in the first step in Figure \ref{oc-reverse-homology-link-fig1}. Then, using fork sliding (Part (1) of Theorem \ref{thm-MOY-knotted-invariance}) and Reidemeister moves of type (II), we can push the left vertex of this bubble along $K$ until it is back in the same small segment of $K$. This is shown in step two in Figure \ref{oc-reverse-homology-link-fig1}. The last two steps in Figure \ref{oc-reverse-homology-link-fig1} are local and self-explanatory.  The end result of all these changes is a link $L_1$ consisting of $L'$ and an extra component $\widetilde{K}$ colored by $N$ that is obtained by slightly pushing $K$ in the direction of its framing. (So $\widetilde{K}$ is isotopic to $K$.) By Theorems \ref{thm-MOY-calculus}, \ref{thm-MOY-knotted-invariance} and Lemma \ref{lemma-l-N-crossings}, each step in Figure \ref{oc-reverse-homology-link-fig1} induces a homogeneous homotopy equivalence that preserves both the quantum grading and the homological grading. So 
\[
\hat{C}_N(L_1) \simeq \hat{C}_N(L).
\]

By switching the upper- and lower-branches at crossings, we can unlink $\widetilde{K}$ from every component of $L'$. From relations \eqref{eq-l-N--to+} and \eqref{eq-l-N-+to-} in Lemma \ref{lemma-l-N-crossings}, we know that unlinking $\widetilde{K}$ from a component $K'$ of $L'$ shifts the homological grading by $- 2 \mathsf{c}(K') l(K,K')$ and the quantum grading by $2 \mathsf{c}(K') l(K,K')$. Note that:
\begin{itemize}
	\item If $K'$ is the component of $L'$ obtained by reversing the orientation and the color of $K$, then $\mathsf{c}(K') = N- \mathsf{c}(K)$ and $l(K,K') = -w(K)$.
	\item If $K'$ is any other component of $L'$, then $K'$ is also a component of $L$. More precisely, $K' \in \mathcal{K}\setminus \{K\}$.
\end{itemize}
Thus, unlinking $\widetilde{K}$ from $L'$ shifts the homological grading by $$2(N-\mathsf{c}(K))w(K) - 2\sum_{K' \in \mathcal{K}\setminus \{K\}} \mathsf{c}(K') l(K,K') = Nw(K) + s$$ and the quantum grading by $$-2(N-\mathsf{c}(K))w(K) + 2\sum_{K' \in \mathcal{K}\setminus \{K\}} \mathsf{c}(K') l(K,K') = -Nw(K) - s.$$
In other words, we have
\[
\hat{C}_N(L' \sqcup \widetilde{K}) \simeq \hat{C}_N(L_1) ~\| Nw(K) + s \|~ \{ q^{-Nw(K) - s} \},
\]
where $L' \sqcup \widetilde{K}$ is $L'$ plus a copy of $\widetilde{K}$ that is unlinked to $L'$.

Next, using \eqref{eq-l-N-+} and \eqref{eq-l-N--} in Lemma \ref{lemma-l-N-crossings}, we can change $\widetilde{K}$ (which is now not linked to $L'$) into an unlink $U$ with Seifert framing (which is not linked to $L'$) and get 
\[
\hat{C}_N ( U ) \simeq \hat{C}_N (\widetilde{K}) \| -Nw(K) \| \{ q^{Nw(K)}\}.
\]
Putting the above together, we get
\begin{equation}\label{eq-oc-reverse-homology-link-proof-1}
\hat{C}_N(L' \sqcup U) \simeq \hat{C}_N(L) ~\| s \|~ \{ q^{-s}\}.
\end{equation}
Finally, by Part (1) of Theorem \ref{thm-MOY-calculus}, we have
\begin{equation}\label{eq-oc-reverse-homology-link-proof-2}
\hat{C}_N(L') \simeq \hat{C}_N(L' \sqcup U).
\end{equation}
Homotopy equivalence \eqref{eq-oc-reverse-homology-link} follows from \eqref{eq-oc-reverse-homology-link-proof-1} and \eqref{eq-oc-reverse-homology-link-proof-2}.
\end{proof}

For a knotted MOY graph, we also have a notion of simple circuits. For a knotted MOY graph $D$, a subgraph $\Delta$ of $D$ is called a simple circuit if
\begin{enumerate}[(i)]
	\item $\Delta$ is a diagram of a knot;
	\item the orientations of all edges of $\Delta$ coincide with the same orientation of this knot. 
\end{enumerate}

Combining the tricks used in the proofs of Theorems \ref{thm-oc-reverse-homology} and \ref{thm-oc-reverse-homology-link}, one can prove the following corollary. We leave the proof to the reader.

\begin{corollary}\label{cor-oc-reverse-homology-general}
Let $D$ be a knotted MOY graph and $\Delta$ a simple circuit of $D$. Denote by $D'$ the knotted MOY graph obtained from $D$ by reversing the orientation and the color of edges along $D$. Then, up to an overall shift of the $\zed_2$-, quantum and homological gradings, $\hat{C}_N(D)$ is homotopic to $\hat{C}_N(D')$. 

In particular, $\left\langle D \right\rangle_N$ and $\left\langle D' \right\rangle_N$ differ from each other only by a factor of the form $\pm q^k$.
\end{corollary}

\section{An Explicit $\mathfrak{so}(6)$ Kauffman Homology}

Theorem \ref{thm-oc-reverse-homology-link} implies that the $N$-colored $\mathfrak{sl}(2N)$ link homology is essentially an invariant of unoriented links. If $N=1$, this homology is the well known Khovanov homology \cite{K1}. In this section, we shed some light on the $2$-colored $\mathfrak{sl}(4)$ link homology. Specifically, we use results from the first two sections to verify that, up to normalization, the $2$-colored $\mathfrak{sl}(4)$ Reshetikhin-Turaev link polynomial is equal to the $\mathfrak{so}(6)$ Kauffman polynomial and, therefore, up to normalization, the $2$-colored $\mathfrak{sl}(4)$ link homology categorifies the $\mathfrak{so}(6)$ Kauffman polynomial. We do so by comparing the Jaeger Formula for the $\mathfrak{so}(6)$ KV polynomial (equation \eqref{eq-Jaeger-formula-N-graph} with $N=3$) to the composition product of the MOY polynomial associated to the embedding $\mathfrak{sl}(1)\times\mathfrak{sl}(3)\hookrightarrow \mathfrak{sl}(4)$. 

\begin{remark}\label{remark-approaches}
There is an alternative approach to the above result. Basically, one can apply Corollary \ref{cor-oc-reverse-homology-general} to the $\mathfrak{sl}(4)$ MOY polynomial of MOY resolutions of $2$-colored link diagrams and keep track of the shifting of the quantum grading while doing so. This would allow one to show that the $2$-colored $\mathfrak{sl}(4)$ Reshetikhin-Turaev link polynomial satisfies the skein relation \eqref{Kauffman-skein} for the $\mathfrak{so}(6)$ Kauffman polynomial. We leave it to the reader to figure out the details of this approach.

Since the coincidence of the $\mathfrak{so}(6)$ Jaeger Formula and the $\mathfrak{sl}(1)\times\mathfrak{sl}(3)\hookrightarrow \mathfrak{sl}(4)$ composition product is itself interesting, we choose to use this coincidence in our proof. 
\end{remark}

\subsection{Renormalizing the $N$-colored $\mathfrak{sl}(2N)$ link homology} We start by renormalizing the $N$-colored $\mathfrak{sl}(2N)$ link homology to make it independent of the orientation.

\begin{definition}\label{def-2N-homology-renormalized}
Let $L$ be an oriented framed link that is colored entirely by $N$. Assume the writhe of $L$ is $w(L)$. We define
\begin{eqnarray}
\label{eq-renormal-homology} \widetilde{C}_{2N}(L) & = & \hat{C}_{2N}(L) \|-\frac{N}{2}w(L)\| \{q^{\frac{N}{2}w(L)}\}, \\
\label{eq-renormal-polynomial} \widetilde{R}_{2N}(L) & = & (-q)^{\frac{N}{2}w(L)} \left\langle L \right\rangle_{2N}.
\end{eqnarray}
Denote by $\widetilde{H}_{2N}(L)$ the homology of $\widetilde{C}_{2N}(L)$.

Note that, if $N$ is odd, then $\frac{N}{2}w(L) \in \frac{1}{2}\zed$. In this case, $\widetilde{C}_{2N}(L)$ and $\widetilde{H}_{2N}(L)$ are $(\frac{1}{2}\zed)\oplus(\frac{1}{2}\zed)$-graded.
\end{definition}

\begin{lemma}\label{lemma-independence-orientation}
The homotopy type of $\widetilde{C}_{2N}(L)$, with its quantum and homological gradings, is independent of the orientation of $L$.
Consequently, $\widetilde{R}_{2N}(L)$ does not depend on the orientation of $L$.
\end{lemma}
\begin{proof}
This follows easily from Theorem \ref{thm-oc-reverse-homology-link}.
\end{proof}

\subsection{The composition product} Now we review the composition product established in \cite{Wu-color-MFW}.

\begin{definition}\label{def-MOY-label}
Let $\Gamma$ be a MOY graph. Denote by $\mathsf{c}$ its color function. That is, for every edge $e$ of $\Gamma$, the color of $e$ is $\mathsf{c}(e)$. A labeling $\mathsf{f}$ of $\Gamma$ is a MOY coloring of the underlying oriented trivalent graph of $\Gamma$ such that $\mathsf{f}(e)\leq \mathsf{c}(e)$ for every edge $e$ of $\Gamma$.

Denote by $\mathcal{L}(\Gamma)$ the set of all labellings of $\Gamma$. For every $\mathsf{f} \in \mathcal{L}(\Gamma)$, denote by $\Gamma_{\mathsf{f}}$ the MOY graph obtained by re-coloring the underlying oriented trivalent graph of $\Gamma$ using $\mathsf{f}$.

For every $\mathsf{f} \in \mathcal{L}(\Gamma)$, define a function $\bar{\mathsf{f}}$ on $E(\Gamma)$ by $\bar{\mathsf{f}}(e)= \mathsf{c}(e)- \mathsf{f}(e)$ for every edge $e$ of $\Gamma$. It is easy to see that $\bar{\mathsf{f}}\in \mathcal{L}(\Gamma)$.

Let $v$ be a vertex of $\Gamma$ of either type in Figure \ref{fig-MOY-vertex}. (Note that, in either case, $e_1$ is to the left of $e_2$ when one looks in the direction of $e$.) For every $\mathsf{f} \in \mathcal{L}(\Gamma)$, define
\[
[v|\Gamma|\mathsf{f}] = \frac{1}{2} (\mathsf{f}(e_1)\bar{\mathsf{f}}(e_2) - \bar{\mathsf{f}}(e_1)\mathsf{f}(e_2)).
\]
\end{definition}

The following is the composition product established in \cite{Wu-color-MFW}.

\begin{theorem}\cite[Theorem 1.7]{Wu-color-MFW}\label{THM-composition-product}
Let $\Gamma$ be a MOY graph. For positive integers $M$, $N$ and $\mathsf{f} \in \mathcal{L}(\Gamma)$, define
\[
\sigma_{M,N}(\Gamma,\mathsf{f}) = M \cdot \mathrm{rot}(\Gamma_{\bar{\mathsf{f}}}) - N \cdot \mathrm{rot}(\Gamma_{\mathsf{f}}) +  \sum_{v\in V(\Gamma)} [v|\Gamma|\mathsf{f}],
\] 
where the rotation number $\rot$ is defined in \eqref{eq-rot-gamma}. Then
\begin{equation}\label{eq-composition-product}
\left\langle \Gamma \right\rangle_{M+N} = \sum_{\mathsf{f} \in \mathcal{L}(\Gamma)} q^{\sigma_{M,N}(\Gamma,\mathsf{f})} \cdot \left\langle \Gamma_{\mathsf{f}} \right\rangle_M \cdot \left\langle \Gamma_{\bar{\mathsf{f}}} \right\rangle_N.
\end{equation}
\end{theorem}

\begin{remark}
The composition product \eqref{eq-composition-product} can be viewed as induced by the embedding $\mathfrak{su}(M)\times\mathfrak{su}(N)\hookrightarrow \mathfrak{su}(M+N)$. The Jaeger Formula \eqref{eq-Jaeger-formula-N} and \eqref{eq-Jaeger-formula-N-graph} can be viewed as induced by the embedding $\mathfrak{su}(N)\hookrightarrow \mathfrak{so}(2N)$. In \cite{Chen-Reshetikhin}, Chen and Reshetikhin presented an extensive study of formulas of the (uncolored) HOMFLY-PT and Kauffman polynomials induced by these and other embeddings.
\end{remark}

\subsection{The $\mathfrak{so}(6)$ KV polynomial and the $\mathfrak{sl}(4)$ MOY polynomial} In this section, we prove that the $\mathfrak{so}(6)$ KV polynomial of a planar $4$-valent graph is equal to the $\mathfrak{sl}(4)$ MOY polynomial of ``mostly $2$-colored" MOY graphs.

\begin{figure}[ht]
\vspace{-1pc}
\[
\xymatrix{
 \input{wide-edge-2} &&  \input{box} \\
 \text{Type one } && \text{Type two}
}
\]
\caption{}\label{mostly-2-colored-MOY-local-figure}

\end{figure}

\begin{definition}\label{def-mostly-2-colored-MOY}
We call a MOY graph mostly $2$-colored if all of its edges are colored by $2$ except in local configurations of the two types given in Figure \ref{mostly-2-colored-MOY-local-figure}. 

Note that all vertices of a mostly $2$-colored MOY graph are contained in such local configurations.
\end{definition}

\begin{figure}[ht]
\vspace{-1pc}
\[
\xymatrix{
 \input{wide-edge-2} \ar@{~>}[rr] &&  \input{edges-tb} \\
 \input{box} \ar@{~>}[rr] && \input{vertex-unoriented-2} \\
 \Gamma && G(\Gamma)
}
\]
\caption{}\label{mostly-2-colored-MOY-to-4-valent-figure}

\end{figure}

The following is the main theorem of this subsection.

\begin{theorem}\label{thm-mostly-2-colored-MOY-to-4-valent}
Given a mostly $2$-colored MOY graph $\Gamma$, we remove the $4$-colored edge in every type one local configuration, shrink the square in every type two configuration to a vertex, remove color and orientation, and smooth out all the vertices of valence $2$. (See Figure \ref{mostly-2-colored-MOY-to-4-valent-figure}.) This gives an unoriented $4$-valent graph $G(\Gamma)$ embedded in $\mathbb{R}^2$. Then
\begin{equation}\label{eq-mostly-2-colored-MOY-to-4-valent}
\left\langle \Gamma \right\rangle_4 = P_6(G(\Gamma)).
\end{equation}
\end{theorem}

Note that the $\mathfrak{so}(6)$ Jaeger formula expresses $P_6(G(\Gamma))$ as a state sum of $\mathfrak{sl}(3)$ MOY polynomials and the $\mathfrak{sl}(1)\times\mathfrak{sl}(3)\hookrightarrow \mathfrak{sl}(4)$ composition product expresses $\left\langle \Gamma \right\rangle_4$ as a state sum of $\mathfrak{sl}(3)$ MOY polynomials. We prove equation \eqref{eq-mostly-2-colored-MOY-to-4-valent} by showing that these two state sums are essentially the same. Several notions of rotation numbers are involved in these two formulas. We need Lemma \ref{lemma-rotation-numbers-oc-reverse} below to track the rotation numbers of MOY graphs.

\begin{figure}[ht]
\vspace{-1pc}
\[
\xymatrix{
v && v' && d(v \leadsto v') \\
\input{MOY-vertex-2-no-edge-names} && \input{MOY-vertex-4-no-edge-names}   &&  \frac{n}{2} \\
\input{MOY-vertex-2-no-edge-names} && \input{MOY-vertex-6-no-edge-names}   &&  -\frac{m}{2} \\
\input{MOY-vertex-1-no-edge-names} && \input{MOY-vertex-3-no-edge-names}   &&  \frac{n}{2} \\
\input{MOY-vertex-1-no-edge-names} && \input{MOY-vertex-5-no-edge-names}   &&  -\frac{m}{2} 
}
\]
\caption{}\label{rotation-numbers-oc-reverse-index-fig}

\end{figure}

\begin{lemma}\label{lemma-rotation-numbers-oc-reverse}
Let $\Gamma$ be a MOY graph and $\Delta$ a simple circuit of $\Gamma$. Denote by $\Gamma'$ the MOY graph obtained from $\Gamma$ by reversing the orientation and the color of edges (with respect to $N$) along $\Delta$. Recall that the rotation number of a MOY graph is defined in equation \eqref{eq-rot-gamma}. We view $\Delta$ as an uncolored oriented circle embedded $\mathbb{R}^2$ and define $\rot\Delta$ to be the usual rotation number of this circle. Then 
\begin{equation}\label{eq-rotation-numbers-oc-reverse}
\rot \Gamma' = \rot \Gamma -N \rot \Delta + \sum_v d(v \leadsto v'),
\end{equation}
where $v$ runs through all vertices of $\Delta$, and $d(v\leadsto v')$ is defined in Figure \ref{rotation-numbers-oc-reverse-index-fig}.
\end{lemma}

\begin{proof}
We prove Lemma \ref{lemma-rotation-numbers-oc-reverse} using a localization of the rotation number similar to that used in the proof of Theorem \ref{thm-oc-reverse}.

Cut each edge of $\Gamma$ at one point in its interior. This divides $\Gamma$ into a collection of neighborhoods of its vertices, each of which is a vertex with three adjacent half-edges. (See Figure \ref{fig-MOY-vertex-angles}, where $e$, $e_1$ and $e_2$ are the three half-edges.)

For a vertex of $\Gamma$, if it is of the form $v$ in Figure \ref{fig-MOY-vertex-angles}, we denote by $\alpha$ the directed angle from $e_1$ to $e$ and by $\beta$ the directed angle from $e_2$ to $e$. We define
\begin{equation}\label{gamma-rot-def-local-v}  
\rot(v) = \frac{m+n}{2\pi} \int_{e}\kappa ds  +\frac{m}{2\pi}\left(\alpha + \int_{e_1}\kappa ds\right)   + \frac{n}{2\pi}\left(\beta + \int_{e_2}\kappa ds\right),
\end{equation}
where $\kappa$ is the signed curvature of a plane curve.

If the vertex is of the form $\hat{v}$ in Figure \ref{fig-MOY-vertex-angles}, we denote by $\hat{\alpha}$ the directed angle from $e$ to $e_1$ and by $\hat{\beta}$ the directed angle from $e$ to $e_2$. We define
\begin{equation}\label{gamma-rot-def-local-v-prime}  
\rot(\hat{v}) = \frac{m+n}{2\pi}\int_{e}\kappa ds + \frac{m}{2\pi} \left(\hat{\alpha}+\int_{e_1}\kappa ds\right) + \frac{n}{2\pi}\left(\hat{\beta}+\int_{e_2}\kappa\right) ds . 
\end{equation}

By the Gauss-Bonnet Theorem, one can easily see that 
\begin{equation} \label{eq-gamma-rot-sum}
\rot(\Gamma) = \sum_{v \in V(\Gamma)} \rot(v).
\end{equation}

\begin{figure}[ht]
\vspace{-1pc}
\[
\xymatrix{
\input{Delta-vertex}
}
\]
\caption{}\label{delta-rot-local-fig}

\end{figure}

For a vertex $v$ of $\Delta$, denote by $e_1$ and $e_2$ the two half-edges incident at $v$ belonging to $\Delta$. Assume that $e_1$ points into $v$, $e_2$ points out of $v$, and the directed angle from $e_1$ to $e_2$ is $\theta$. (See Figure \ref{delta-rot-local-fig}.) Define
\begin{equation}\label{eq-delta-rot-local}
\rot_\Delta (v) = \frac{1}{2\pi}\left(\int_{e_1}\kappa ds +\theta + \int_{e_2}\kappa ds\right).
\end{equation}
By the Gauss-Bonnet Theorem, we know $\rot \Delta = \sum_v \rot_\Delta (v)$, where $v$ runs through all vertices of $\Delta$.

For a vertex $v$ of $\Gamma$ contained in $\Delta$, denote by $v'$ the vertex of $\Gamma'$ corresponding to $v$. We claim
\begin{equation}\label{eq-local-rot-change}
\rot(v') = \rot(v)- N \rot_\Delta (v)  + d(v \leadsto v').
\end{equation}
Clearly, the lemma follows from \eqref{eq-local-rot-change}.

To prove \eqref{eq-local-rot-change}, one needs to check that it is true for all four cases listed in Figure \ref{rotation-numbers-oc-reverse-index-fig}. Since the proofs in all four cases are very similar, we only check the first case here and leave the other three to the reader.

In the first case, $v$ and $v'$ are depicted in Figure \ref{fig-MOY-vertex-change}. As before, denote by $\alpha$ the directed angle from $e_1$ to $e$, by $\beta$ the directed angle from $e_2$ to $e$ and by $\gamma$ the directed angle from $e_2'$ to $e_1'$. Then
\begin{eqnarray*}
\rot(v') & =  & \frac{N-m}{2 \pi} \int_{e_1'}\kappa ds + \frac{N-m-n}{2\pi} \left(-\alpha + \int_{e'}\kappa ds\right) + \frac{n}{2\pi} \left(\gamma + \int_{e_2'}\kappa ds \right) \\
& = & \frac{m}{2\pi} \left(\alpha + \int_{e_1}\kappa ds \right) + \frac{n}{2\pi} \left(\beta + \int_{e_2}\kappa ds \right) + \frac{m+n}{2\pi} \left(\int_{e}\kappa ds\right) \\
&& - \frac{N}{2\pi} \left( \int_{e_1}\kappa ds + \alpha + \int_{e}\kappa ds\right) + \frac{\alpha + \gamma -\beta}{2\pi} n \\
& = & \rot(v)- N \rot_\Delta (v)  + \frac{n}{2},
\end{eqnarray*}
where, in the last step, we used the fact that $\alpha + \gamma -\beta = \pi$. This proves \eqref{eq-local-rot-change} in the first case in Figure \ref{rotation-numbers-oc-reverse-index-fig}.
\end{proof}

Now we are ready to prove Theorem \ref{thm-mostly-2-colored-MOY-to-4-valent}.

\begin{proof}[Proof of Theorem \ref{thm-mostly-2-colored-MOY-to-4-valent}]
By the composition product \eqref{eq-composition-product}, we know that
\[
\left\langle \Gamma \right\rangle_4 = \sum_{\mathsf{f} \in \mathcal{L}(\Gamma)} q^{\sigma_{1,3}(\Gamma,\mathsf{f})} \cdot \left\langle \Gamma_{\mathsf{f}} \right\rangle_1 \cdot \left\langle \Gamma_{\bar{\mathsf{f}}} \right\rangle_3,
\]
where
\begin{equation}\label{eq-composition-product-1+3-power-label}
\sigma_{1,3}(\Gamma,\mathsf{f}) = \mathrm{rot}(\Gamma_{\bar{\mathsf{f}}}) - 3 \cdot \mathrm{rot}(\Gamma_{\mathsf{f}}) +  \sum_{v\in V(\Gamma)} [v|\Gamma|\mathsf{f}].
\end{equation}
Note that, in order for the product $\left\langle \Gamma_{\mathsf{f}} \right\rangle_1 \cdot \left\langle \Gamma_{\bar{\mathsf{f}}} \right\rangle_3$ to be non-zero, we must have $\mathsf{f}(e)=0,1$ and $0\leq\bar{\mathsf{f}}(e) \leq 3$ for all edges of $\Gamma$. Define 
\[
\mathcal{L}_{\neq0}(\Gamma) = \{\mathsf{f} \in  \mathcal{L}(\Gamma)~|~ \mathsf{f}(e)=0,1,~ 0\leq\bar{\mathsf{f}}(e) \leq 3 ~\forall e \in E(\Gamma)\}.
\]
Then
\begin{equation}\label{eq-composition-product-1+3}
\left\langle \Gamma \right\rangle_4 = \sum_{\mathsf{f} \in \mathcal{L}_{\neq0}(\Gamma)} q^{\sigma_{1,3}(\Gamma,\mathsf{f})}  \cdot \left\langle \Gamma_{\bar{\mathsf{f}}} \right\rangle_3,
\end{equation}
where we used the fact that $\left\langle \Gamma_{\mathsf{f}} \right\rangle_1=1$ for all $\mathsf{f} \in  \mathcal{L}_{\neq0}(\Gamma)$.

Denote by $E_2(\Gamma)$ the set of edges of $\Gamma$ colored by $2$ and by $E(G(\Gamma))$ the set of edges of $G(\Gamma)$. Then there is a surjective function $g:E_2(\Gamma) \rightarrow E(G(\Gamma))$ such that $g(e)$ is the edge of $G(\Gamma)$ ``containing" $e$ for every $e\in E_2(\Gamma)$.

Note that, for any $\mathsf{f} \in  \mathcal{L}_{\neq0}(\Gamma)$, $\Gamma_{\mathsf{f}}$ is a collection of pairwise disjoint embedded circles colored by $1$ (after erasing edges colored by $0$.) All possible intersections of $\Gamma_{\mathsf{f}}$ with type one and type two local configurations (defined in Figure \ref{mostly-2-colored-MOY-local-figure}) are described in Figures \ref{local-weights-type-1-fig} and \ref{local-weights-type-2-fig} in Appendix \ref{app-figures}, where edges belonging to $\Gamma_{\mathsf{f}}$ are traced out by \textcolor{BrickRed}{red} paths.

For any $\mathsf{f} \in  \mathcal{L}_{\neq0}(\Gamma)$, we define an edge orientation $\varrho_{\mathsf{f}}$ of $G(\Gamma)$ such that, for all $e \in E_2(\Gamma)$,
\[
\varrho_{\mathsf{f}}(g(e)) = \begin{cases}
\text{the orientation of } e & \text{if } \mathsf{f}(e) =1, \\
\text{the opposite of the orientation of } e & \text{if } \mathsf{f}(e) =0.
\end{cases}
\]
It is easy to check that $\varrho_{\mathsf{f}}(g(e))$ is a well define balanced edge orientation of $G(\Gamma)$ and the mapping $\mathsf{f} \mapsto \varrho_{\mathsf{f}}$ is a surjection $\mathcal{L}_{\neq0}(\Gamma) \rightarrow \mathcal{O}(G(\Gamma))$.

Given a labeling $\mathsf{f} \in  \mathcal{L}_{\neq0}(\Gamma)$, we define a subgraph $\Delta_{\mathsf{f}}$ of $\Gamma$ (and therefore of $\Gamma_{\bar{\mathsf{f}}}$) such that
\begin{itemize}
	\item if $e \in E_2(\Gamma)$, then $e$ is in $\Delta_{\mathsf{f}}$ if and only if $\mathsf{f}(e)=0$;
	\item edges of $\Delta_{\mathsf{f}}$ of other colors (which are all contained in local configurations of type one or two) are traced out by \textcolor{SkyBlue}{blue} paths in Figures \ref{local-weights-type-1-fig} and \ref{local-weights-type-2-fig} in Appendix \ref{app-figures}.
\end{itemize}
It is easy to see that $\Delta_{\mathsf{f}}$ is a union of pairwise disjoint simple circuits of $\Gamma$ (and therefore of $\Gamma_{\bar{\mathsf{f}}}$.) Reversing the orientation and the color (with respect to $3$) of edges of $\Gamma_{\bar{\mathsf{f}}}$ along $\Delta_{\mathsf{f}}$, we get a MOY graph $\Gamma_{\bar{\mathsf{f}}}'$. By Theorem \ref{thm-oc-reverse}, we have
\begin{equation}\label{eq-Gamma-bar-Gamma-bar-prime}
\left\langle \Gamma_{\bar{\mathsf{f}}}\right\rangle_3 = \left\langle  \Gamma_{\bar{\mathsf{f}}}'  \right\rangle_3.
\end{equation}
In $\Gamma_{\bar{\mathsf{f}}}'$, delete all edges colored by $0$, contract all edge colored by $2$ and smooth out all vertices of valence $2$. This gives a partial resolution $G(\Gamma)_{\varrho_{\mathsf{f}},\varsigma_{\mathsf{f}}}$ of $G(\Gamma)_{\varrho_{\mathsf{f}}}$, where $\varsigma_{\mathsf{f}}$ resolves all vertices of $G(\Gamma)_{\varrho_{\mathsf{f}}}$ except those corresponding to the last case in Figure \ref{local-weights-type-2-fig}. In this last case, we need to further choose the resolution. Note that, by \cite[Lemma 2.4]{MOY}, we have
\begin{equation}\label{eq-MOY-3-3}
\left\langle \input{box-ll-rllr-2-slides} \right\rangle_3 = \left\langle \setlength{\unitlength}{1.75pt}
\begin{picture}(20,10)(-10,7)
\qbezier(-10,0)(0,10)(-10,20)

\qbezier(10,0)(0,10)(10,20)

\put(-10,20){\vector(-1,1){0}}

\put(10,0){\vector(1,-1){0}}

\put(-5,15){\tiny{$1$}}

\put(3,3){\tiny{$1$}}

\end{picture} \right\rangle_3 + \left\langle \setlength{\unitlength}{1.75pt}
\begin{picture}(20,10)(-10,7)
\qbezier(-10,0)(0,10)(10,0)

\qbezier(-10,20)(0,10)(10,20)

\put(-10,20){\vector(-1,1){0}}

\put(10,0){\vector(1,-1){0}}

\put(-5,12){\tiny{$1$}}

\put(3,6){\tiny{$1$}}

\end{picture} \right\rangle_3. \vspace{1pc}
\end{equation}

For $\mathsf{f} \in  \mathcal{L}_{\neq0}(\Gamma)$ and $\varsigma \in \Sigma(G(\Gamma)_{\varrho_{\mathsf{f}}})$, we say that $\varsigma$ is compatible with $\mathsf{f}$ if, on all the vertices that are resolved by $\varsigma_{\mathsf{f}}$, $\varsigma$ agrees with $\varsigma_{\mathsf{f}}$. Denote by $\Sigma_{\mathsf{f}}(G(\Gamma)_{\varrho_{\mathsf{f}}})$ the set of resolutions of $G(\Gamma)_{\varrho_{\mathsf{f}}}$ that are compatible with $\mathsf{f}$. Then the mapping $(\mathsf{f}, \varsigma) \mapsto (\varrho_{\mathsf{f}},\varsigma)$ gives a bijection 
\begin{equation}\label{bijection-labeling-orientation}
\{(\mathsf{f}, \varsigma)~|~ \mathsf{f} \in  \mathcal{L}_{\neq0}(\Gamma),~ \varsigma \in \Sigma_{\mathsf{f}}(G(\Gamma)_{\varrho_{\mathsf{f}}})\} \rightarrow \{(\varrho,\varsigma) ~|~ \varrho \in \mathcal{O}(G(\Gamma)),~ \varsigma \in \Sigma (G(\Gamma)_\varrho\}.
\end{equation}

Let $C$ be a local configuration (of type one or type two in Figure \ref{mostly-2-colored-MOY-local-figure}) in $\Gamma$. For $\mathsf{f} \in  \mathcal{L}_{\neq0}(\Gamma)$  and $\varsigma \in \Sigma_{\mathsf{f}}(G(\Gamma)_{\varrho_{\mathsf{f}}})$, we define two indices $t_{\mathsf{f},\varsigma}(C)$ and $r_{\mathsf{f},\varsigma}(C)$. The values of these two indices are given in Figures \ref{local-weights-type-1-fig} and \ref{local-weights-type-2-fig} in Appendix \ref{app-figures}. It is straightforward to check that
\begin{eqnarray}
\label{eq-rot-change-Gamma-prime-G} \rot (G(\Gamma)_{\varrho_{\mathsf{f}},\varsigma}) & = & \rot (\Gamma_{\bar{\mathsf{f}}}') + \sum_C t_{\mathsf{f},\varsigma}(C), \\
\label{eq-rot-change-rot-f-G} \rot(G(\Gamma)_{\varrho_{\mathsf{f}},\varsigma}) & = & \rot (\Gamma_{\mathsf{f}}) - \rot (\Delta_{\mathsf{f}}) + \sum_C r_{\mathsf{f},\varsigma}(C),
\end{eqnarray}
where $C$ runs through all local configurations of type one or two in $\Gamma$.

Combining the above and Lemma \ref{lemma-rotation-numbers-oc-reverse}, we have that, for any $\mathsf{f} \in  \mathcal{L}_{\neq0}(\Gamma)$  and $\varsigma \in \Sigma_{\mathsf{f}}(G(\Gamma)_{\varrho_{\mathsf{f}}})$,
\begin{eqnarray}
\label{eq-sigma-13-simplified-1} && \sigma_{1,3}(\Gamma,\mathsf{f}) \\
& = & \mathrm{rot}(\Gamma_{\bar{\mathsf{f}}}) - 3  \mathrm{rot}(\Gamma_{\mathsf{f}}) +  \sum_{v\in V(\Gamma)} [v|\Gamma|\mathsf{f}] \nonumber \\
& = & \mathrm{rot}(\Gamma_{\bar{\mathsf{f}}}') - 3  (\mathrm{rot}(\Gamma_{\mathsf{f}})- 3 \rot (\Delta_{\mathsf{f}})) + \sum_{v\in V(\Gamma)} ([v|\Gamma|\mathsf{f}] -d(v\leadsto v')) \nonumber \\
& = & -2\rot (G(\Gamma)_{\varrho_{\mathsf{f}},\varsigma}) + \sum_C (3r_{\mathsf{f},\varsigma}(C)- t_{\mathsf{f},\varsigma}(C)) +  \sum_{v\in V(\Gamma)} ([v|\Gamma|\mathsf{f}]-d(v\leadsto v')) \nonumber \\
& = & -2\rot (G(\Gamma)_{\varrho_{\mathsf{f}},\varsigma}) +\sum_C \left( 3r_{\mathsf{f},\varsigma}(C)- t_{\mathsf{f},\varsigma}(C) + \sum_{v\in V(C)} ([v|\Gamma|\mathsf{f}]-d(v\leadsto v')) \right), \nonumber
\end{eqnarray}
where $C$ runs through all local configurations of type one or two in $\Gamma$, $V(C)$ is the set of vertices of $C$, and we use the convention that $d(v\leadsto v')=0$ if the vertex $v$ is unchanged in $\Gamma_{\bar{\mathsf{f}}} \leadsto \Gamma_{\bar{\mathsf{f}}}'$. The values of $r_{\mathsf{f},\varsigma}(C)$, $t_{\mathsf{f},\varsigma}(C)$, $\sum_{v\in V(C)} [v|\Gamma|\mathsf{f}]$ and $\sum_{v\in V(C)}d(v\leadsto v'))$ are recorded in Figures \ref{local-weights-type-1-fig} and \ref{local-weights-type-2-fig} in Appendix \ref{app-figures}. One can verify case by case that
\begin{eqnarray}
\label{eq-sigma-13-simplified-2} 
&& 3r_{\mathsf{f},\varsigma}(C)- t_{\mathsf{f},\varsigma}(C) + \sum_{v\in V(C)} ([v|\Gamma|\mathsf{f}]-d(v\leadsto v')) \\
& = & \begin{cases}
1 & \text{if } C \text{ is of type two and } \varsigma \text{ applies } L  \\
 & \text{to the corresponding vertex in } G(\Gamma)_{\varrho_{\mathsf{f}}}, \\
 & \\
-1 & \text{if } C \text{ is of type two and } \varsigma \text{ applies } R \\
& \text{to the corresponding vertex in } G(\Gamma)_{\varrho_{\mathsf{f}}}, \\
& \\
0 & \text{otherwise.}
\end{cases} \nonumber
\end{eqnarray}

Equations \eqref{eq-sigma-13-simplified-1} and \eqref{eq-sigma-13-simplified-2} imply that, for any $\mathsf{f} \in  \mathcal{L}_{\neq0}(\Gamma)$ and $\varsigma \in \Sigma_{\mathsf{f}}(G(\Gamma)_{\varrho_{\mathsf{f}}})$,
\begin{equation}\label{CP-Jaeger-rot-weight-match}
q^{\sigma_{1,3}(\Gamma,\mathsf{f})} = q^{-2\rot (G(\Gamma)_{\varrho_{\mathsf{f}},\varsigma})} \cdot [G(\Gamma)_{\varrho_{\mathsf{f}}},\varsigma].
\end{equation}

In view of bijection \eqref{bijection-labeling-orientation}, it follows from \eqref{eq-MOY-HOMFLY}, \eqref{eq-composition-product-1+3}, \eqref{eq-Gamma-bar-Gamma-bar-prime}, \eqref{eq-MOY-3-3} and \eqref{CP-Jaeger-rot-weight-match} that
\begin{eqnarray}
\label{Jaeger-CP-MOY-4-2} \left\langle \Gamma \right\rangle_4 & = & \sum_{\mathsf{f} \in \mathcal{L}_{\neq0}(\Gamma)} \sum_{\varsigma \in \Sigma_{\mathsf{f}}(G(\Gamma)_{\varrho_{\mathsf{f}}})} q^{-2\rot (G(\Gamma)_{\varrho_{\mathsf{f}},\varsigma})} \cdot [G(\Gamma)_{\varrho_{\mathsf{f}}},\varsigma]  \cdot R_3( G(\Gamma)_{\varrho_{\mathsf{f}},\varsigma}) \\
&= & \sum_{\varrho \in \mathcal{O}(G(\Gamma))} \sum_{\varsigma \in \Sigma(G(\Gamma)_{\varrho})} q^{-2\rot (G(\Gamma)_{\varrho,\varsigma})} \cdot [G(\Gamma)_{\varrho},\varsigma]  \cdot R_3( G(\Gamma)_{\varrho,\varsigma}). \nonumber
\end{eqnarray}
Comparing the right hand side of \eqref{Jaeger-CP-MOY-4-2} to the Jaeger Formula \eqref{eq-Jaeger-formula-N-graph} in Theorem \ref{thm-Jaeger-formula-graph}, we get $\left\langle \Gamma \right\rangle_4 = P_6(G(\Gamma))$.
\end{proof}

\subsection{An explicit $\mathfrak{so}(6)$ Kauffman homology} Webster \cite{Webster1,Webster2} has categorified, for any simple complex Lie algebra $\mathfrak{g}$, the quantum $\mathfrak{g}$ invariant for links colored by any finite dimensional representations of $\mathfrak{g}$. But his categorification is very abstract. For applications in knot theory, it would help if we have categorifications that are concrete and explicit. For quantum $\mathfrak{sl}(N)$ link invariants, examples of such categorifications can be found in \cite{K1,KR1,Wu-color}. We know much less about explicit categorifications of quantum $\mathfrak{so}(N)$ link invariants. Khovanov and Rozansky \cite{KR3} proposed a categorification of the $\mathfrak{so}(2N)$ Kauffman polynomial. But its invariance under Reidemeister move (III) is still open. They did however point out that the $\mathfrak{so}(4)$ version of their homology is isomorphic to the tensor square of the Khovanov homology \cite{K1} and is therefore a link invariant. More recently, Cooper, Hogancamp and Krushkal \cite{Cooper-Hogancamp-Krushkal} gave an explicit categorification of the $\mathfrak{so}(3)$ Kauffman polynomial.

Theorem \ref{thm-mostly-2-colored-MOY-to-4-valent} allows us to give an explicit categorification of the $\mathfrak{so}(6)$ Kauffman polynomial. More precisely, we have the following theorem.

\begin{theorem}\label{thm-2-4-MOY-6-Kauffman}
Let $L$ be an unoriented framed link in $S^3$. Fix an orientation $\varrho$ of $L$, color all components of $L$ by $2$ and denote the resulted colored oriented framed link $L_\varrho^{(2)}$. Then
\begin{equation}\label{eq-2-4-MOY-6-Kauffman}
\widetilde{R}_{4}(L_\varrho^{(2)}) = (-1)^m P_6(\overline{L}),
\end{equation}
where the polynomial $\widetilde{R}_{4}$ is defined in Definition \ref{def-2N-homology-renormalized}, $m$ is the ($\zed_2$-) number of crossings in $L$, and $\overline{L}$ is the mirror image of $L$, that is, $L$ with the upper- and lower- branches switched at every crossing.

Consequently, the renormalized $\mathfrak{sl}(4)$ homology $\widetilde{H}_{4}(L_\varrho^{(2)})$ categorifies $(-1)^m P_6(\overline{L})$.
\end{theorem}

\begin{figure}[ht]
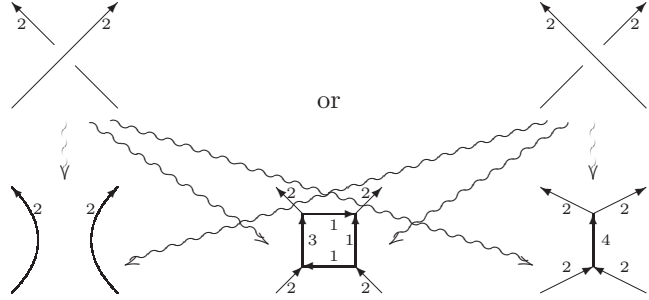

\vspace{-1pc}
\[
\xymatrix{
 \input{crossing-22+} \ar@{~>}[d] \ar@{~>}[drr] \ar@{~>}[drrrr]&& \text{or}  &&  \input{crossing-22-} \ar@{~>}[d] \ar@{~>}[dll] \ar@{~>}[dllll] \\
 \input{regular-smoothing} && \input{box} && \input{wide-edge-2} 
}
\]
\caption{MOY resolutions}\label{MOY-resolutions-figure}

\end{figure}

\begin{figure}[ht]
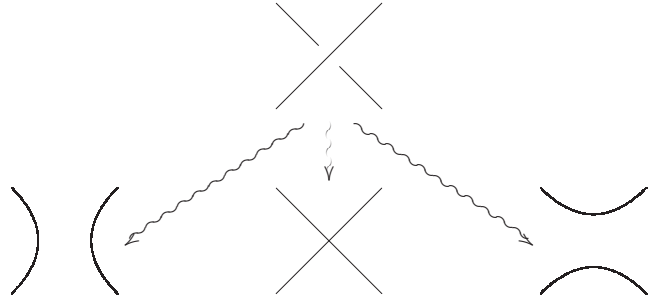

\vspace{-1pc}
\[
\xymatrix{
  &&  \input{crossing-unoriented} \ar@{~>}[d] \ar@{~>}[drr] \ar@{~>}[dll] &&   \\
 \input{regular-smoothing-v} && \input{vertex-unoriented-2} && \input{regular-smoothing-h}
}
\]
\caption{KV resolutions}\label{KV-resolutions-figure}

\end{figure}

\begin{proof}
Replace every crossing of $L_\varrho^{(2)}$ by one of the three local configurations in Figure \ref{MOY-resolutions-figure}. We call the result a MOY resolution of $L_\varrho^{(2)}$ and denote by $\mathcal{MOY}(L_\varrho^{(2)})$ the set of MOY resolutions of $L_\varrho^{(2)}$. Replace every crossing of $L$ by one of the three local configurations in Figure \ref{KV-resolutions-figure}. We call the result a KV resolution of $L$ and denote by $\mathcal{KV}(L)$ the set of KV resolutions of $L$. Note that the mapping $\Gamma \mapsto G(\Gamma)$ defined in Theorem \ref{thm-mostly-2-colored-MOY-to-4-valent} gives a bijection $\mathcal{MOY}(L_\varrho^{(2)}) \rightarrow \mathcal{KV}(L)$.

Compare the skein relations  \vspace{-1pc}
\begin{eqnarray*}
\widetilde{R}_{4} (\setlength{\unitlength}{2pt}
\begin{picture}(20,20)(-10,7)

\put(-10,0){\vector(1,1){20}}

\put(10,0){\line(-1,1){8}}

\put(-2,12){\vector(-1,1){8}}

\put(7,15){\tiny{$2$}}

\put(-9,15){\tiny{$2$}}

\end{picture})  & = & -q^{-1} \widetilde{R}_{4} (\setlength{\unitlength}{2pt}
\begin{picture}(20,20)(-10,7)

\qbezier(-10,0)(0,10)(-10,20)

\qbezier(10,0)(0,10)(10,20)

\put(-10,20){\vector(-1,1){0}}

\put(10,20){\vector(1,1){0}}

\put(4,15){\tiny{$2$}}

\put(-6,15){\tiny{$2$}}

\end{picture}) + \widetilde{R}_{4} (\input{box-slides}) - q \widetilde{R}_{4} (\setlength{\unitlength}{2pt}
\begin{picture}(20,20)(-10,7)
\put(0,15){\vector(2,1){10}}

\put(-10,0){\vector(2,1){10}}

\put(0,15){\vector(-2,1){10}}

\put(10,0){\vector(-2,1){10}}

\put(0,5){\vector(0,1){10}}

\put(1.5,9){\tiny{$4$}}

\put(5,15){\tiny{$2$}}

\put(5,4){\tiny{$2$}}

\put(-6,15){\tiny{$2$}}

\put(-6,4){\tiny{$2$}}

\end{picture}), \\
\widetilde{R}_{4} (\setlength{\unitlength}{2pt}
\begin{picture}(20,20)(-10,7)

\put(-10,0){\vector(1,1){20}}

\put(2,8){\vector(1,-1){8}}

\put(-2,12){\line(-1,1){8}}

\put(7,15){\tiny{$2$}}

\put(-9,15){\tiny{$2$}}

\end{picture}) & = & -q^{-1} \widetilde{R}_{4} (\setlength{\unitlength}{2pt}
\begin{picture}(20,20)(-10,7)
\put(-10,0){\vector(1,2){5}}

\put(-10,20){\vector(1,-2){5}}

\put(-5,10){\vector(1,0){10}}

\put(5,10){\vector(1,2){5}}

\put(5,10){\vector(1,-2){5}}

\put(0,11){\tiny{$4$}}

\put(5,15){\tiny{$2$}}

\put(5,4){\tiny{$2$}}

\put(-6,15){\tiny{$2$}}

\put(-6,4){\tiny{$2$}}

\end{picture}) + \widetilde{R}_{4} (\input{box-slides-r}) - q \widetilde{R}_{4} (\setlength{\unitlength}{2pt}
\begin{picture}(20,20)(-10,7)

\qbezier(-10,0)(0,10)(10,0)

\qbezier(-10,20)(0,10)(10,20)

\put(10,20){\vector(1,1){0}}

\put(10,0){\vector(1,-1){0}}

\put(4,12){\tiny{$2$}}

\put(4,5){\tiny{$2$}}

\end{picture}), \\
P_6 (\setlength{\unitlength}{2pt}
\begin{picture}(20,20)(-10,7)

\put(10,0){\line(-1,1){20}}

\put(-10,0){\line(1,1){8}}

\put(2,12){\line(1,1){8}}

\end{picture}) & = & q^{-1}P_6(\setlength{\unitlength}{2pt}
\begin{picture}(20,20)(-10,7)

\qbezier(-10,0)(0,10)(-10,20)

\qbezier(10,0)(0,10)(10,20)

\end{picture}) - P_6(\setlength{\unitlength}{2pt}
\begin{picture}(20,20)(-10,7)
\put(0,10){\line(1,1){10}}

\put(-10,0){\line(1,1){10}}

\put(0,10){\line(-1,1){10}}

\put(10,0){\line(-1,1){10}}

\end{picture}) + qP_6(\setlength{\unitlength}{2pt}
\begin{picture}(20,20)(-10,7)

\qbezier(-10,0)(0,10)(10,0)

\qbezier(-10,20)(0,10)(10,20)

\end{picture}).
\end{eqnarray*}

\vspace{1pc}

\noindent It is easy to see that equation \eqref{eq-2-4-MOY-6-Kauffman} follows from equation \eqref{eq-mostly-2-colored-MOY-to-4-valent} in Theorem \ref{thm-mostly-2-colored-MOY-to-4-valent}.
\end{proof}

\begin{question}
Is $\widetilde{H}_{4}(L_\varrho^{(2)})$ isomorphic to the $\mathfrak{so}(6)$ version of the homology defined in \cite{KR3}?
\end{question}

\appendix

\section{Figures Used in the Proof of Theorem \ref{thm-mostly-2-colored-MOY-to-4-valent}}\label{app-figures}

In Figures \ref{local-weights-type-1-fig} and \ref{local-weights-type-2-fig}, edges along \textcolor{BrickRed}{red paths} belong to $\Gamma_{\mathsf{f}}$ and edges along \textcolor{SkyBlue}{blue paths} belong to $\Delta_{\mathsf{f}}$.

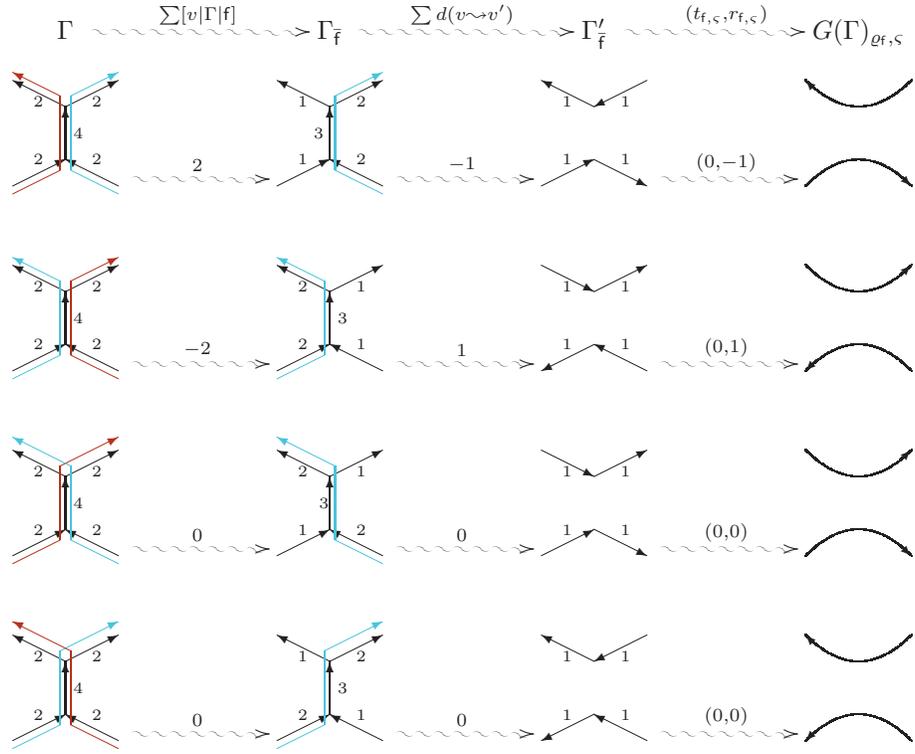
\begin{figure}[ht]
\[
\xymatrix{
\Gamma \ar@{~>}[rr]^{\sum [v|\Gamma|\mathsf{f}]} && \Gamma_{\bar{\mathsf{f}}} \ar@{~>}[rr]^{\sum d(v\leadsto v')} && \Gamma_{\bar{\mathsf{f}}}' \ar@{~>}[rr]^{(t_{\mathsf{f},\varsigma},r_{\mathsf{f},\varsigma})} && G(\Gamma)_{\varrho_{\mathsf{f}},\varsigma} \\
\input{wide-edge-2-ll} \ar@{~>}[rr]^{2} && \input{wide-edge-2-ll-1} \ar@{~>}[rr]^{-1} && \input{wide-edge-2-ll-2} \ar@{~>}[rr]^{(0,-1)} && \input{wide-edge-2-ll-3} \\
\input{wide-edge-2-rr} \ar@{~>}[rr]^{-2} && \input{wide-edge-2-rr-1} \ar@{~>}[rr]^{1} && \input{wide-edge-2-rr-2} \ar@{~>}[rr]^{(0,1)} && \input{wide-edge-2-rr-3} \\
\input{wide-edge-2-lr} \ar@{~>}[rr]^{0} && \input{wide-edge-2-lr-1} \ar@{~>}[rr]^{0} && \input{wide-edge-2-lr-2} \ar@{~>}[rr]^{(0,0)} && \input{wide-edge-2-lr-3} \\
\input{wide-edge-2-rl} \ar@{~>}[rr]^{0} && \input{wide-edge-2-rl-1} \ar@{~>}[rr]^{0} && \input{wide-edge-2-rl-2} \ar@{~>}[rr]^{(0,0)} && \input{wide-edge-2-rl-3}
}
\]
\caption{}\label{local-weights-type-1-fig}

\end{figure}

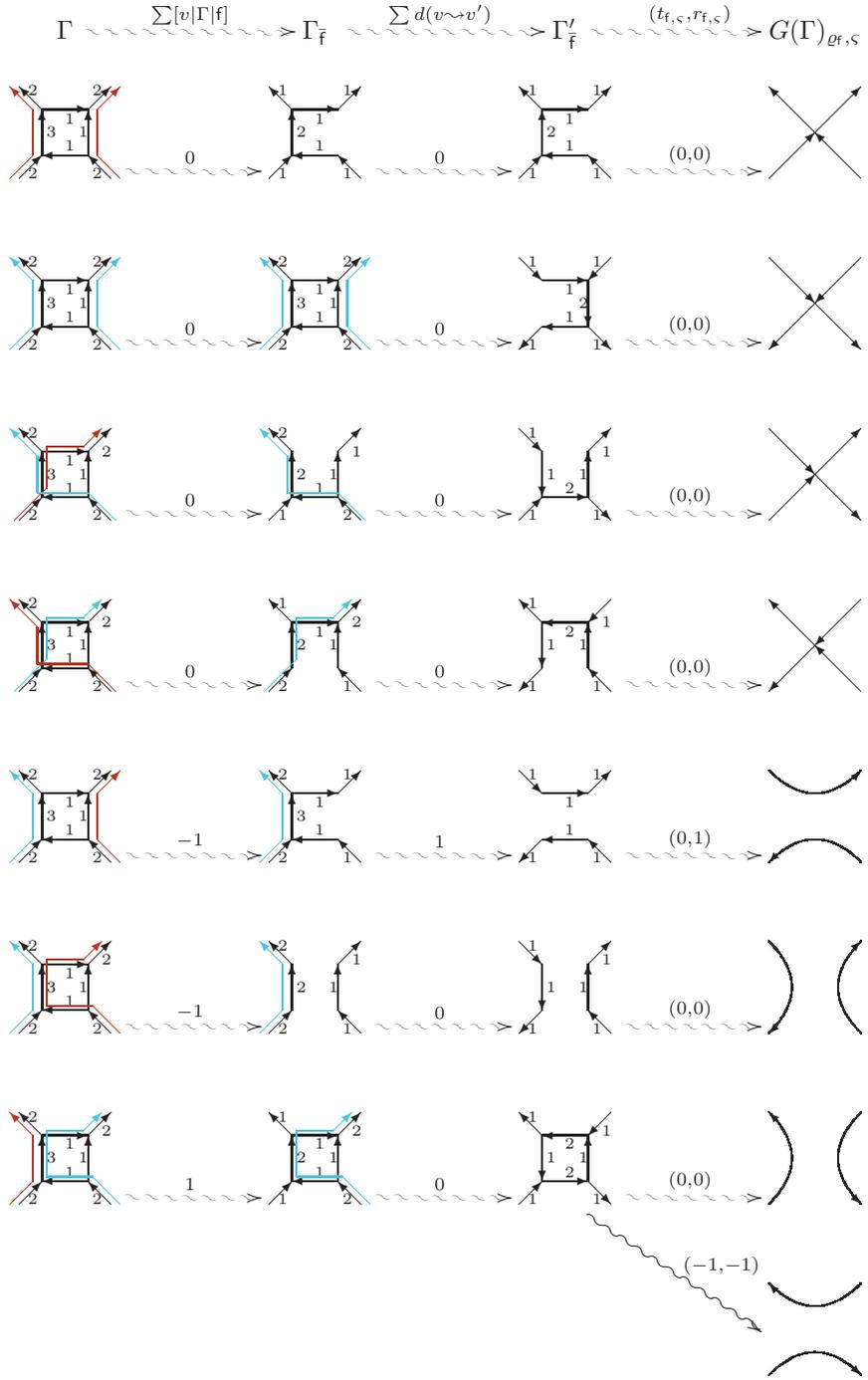
\begin{figure}[ht]
\[
\xymatrix{
\Gamma \ar@{~>}[rr]^{\sum [v|\Gamma|\mathsf{f}]} && \Gamma_{\bar{\mathsf{f}}} \ar@{~>}[rr]^{\sum d(v\leadsto v')} && \Gamma_{\bar{\mathsf{f}}}' \ar@{~>}[rr]^{(t_{\mathsf{f},\varsigma},r_{\mathsf{f},\varsigma})} && G(\Gamma)_{\varrho_{\mathsf{f}},\varsigma} \\
\input{box-llrr-null} \ar@{~>}[rr]^{0} && \input{box-llrr-null-1} \ar@{~>}[rr]^{0} && \input{box-llrr-null-2} \ar@{~>}[rr]^{(0,0)} && \input{box-llrr-null-3} \\
\input{box-null-llrr} \ar@{~>}[rr]^{0} && \input{box-null-llrr-1} \ar@{~>}[rr]^{0} && \input{box-null-llrr-2} \ar@{~>}[rr]^{(0,0)} && \input{box-null-llrr-3} \\
\input{box-lr-rl} \ar@{~>}[rr]^{0} && \input{box-lr-rl-1} \ar@{~>}[rr]^{0} && \input{box-lr-rl-2} \ar@{~>}[rr]^{(0,0)} && \input{box-lr-rl-3} \\
\input{box-rl-lr} \ar@{~>}[rr]^{0} && \input{box-rl-lr-1} \ar@{~>}[rr]^{0} && \input{box-rl-lr-2} \ar@{~>}[rr]^{(0,0)} && \input{box-rl-lr-3} \\
\input{box-rr-ll} \ar@{~>}[rr]^{-1} && \input{box-rr-ll-1} \ar@{~>}[rr]^{1} && \input{box-rr-ll-2} \ar@{~>}[rr]^{(0,1)} && \input{box-rr-ll-3} \\
\input{box-rllr-ll} \ar@{~>}[rr]^{-1} && \input{box-rllr-ll-1} \ar@{~>}[rr]^{0} && \input{box-rllr-ll-2} \ar@{~>}[rr]^{(0,0)} && \input{box-rllr-ll-3} \\
\input{box-ll-rllr} \ar@{~>}[rr]^{1} && \input{box-ll-rllr-1} \ar@{~>}[rr]^{0} && \input{box-ll-rllr-2} \ar@{~>}[rr]^{(0,0)} \ar@{~>}[drr]^{(-1,-1)}  && \input{box-ll-rllr-3-left} \\
&&&&&& \input{box-ll-rllr-3-right}
}
\]
\caption{}\label{local-weights-type-2-fig}

\end{figure}

\end{document}